\documentclass[preprint,10pt,authoryear]{elsarticle}
\usepackage{newtxtext,newtxmath}

\usepackage{dsfont}
\usepackage{amsthm,mathtools}
\usepackage{graphicx}
\usepackage{placeins}
\usepackage{needspace}
\usepackage{xcolor}
\usepackage{etoolbox}
\usepackage[hidelinks]{hyperref}
\usepackage{xspace}
\usepackage[margin=1in]{geometry}
\journal{European Journal of Mechanics - A/Solids}

\definecolor{anncolor}{RGB}{0,0,0}
\definecolor{newanncolor}{RGB}{0,0,0}
\colorlet{revisionred}{anncolor}
\DeclareRobustCommand{\newann}[1]{{\color{newanncolor}#1}}

\newcommand{\ann}[1]{{\color{anncolor}#1}}
\newcommand{\revisionemailcolor}[1]{%
  \ifstrequal{#1}{2512255@tongji.edu.cn}{\color{anncolor}}{%
  \ifstrequal{#1}{2610189@tongji.edu.cn}{\color{anncolor}}{%
  \ifstrequal{#1}{patrizio.neff@uni-due.de}{\color{anncolor}}{}}}}
\patchcmd{\emailauthor}{\ttfamily}{\ttfamily\revisionemailcolor{#1}}{}{\errmessage{Email colour patch failed}}
\newcommand{\TSTSMplain}{TSTS-M\textsuperscript{++}}
\newcommand{\TSTSM}{\texorpdfstring{\TSTSMplain}{TSTS-M++}\xspace}
\newcommand{\Diff}{\mathrm{D}}
\newcommand{\Id}{\mathds{1}}
\newcommand{\Cof}{\operatorname{Cof}}
\newcommand{\annbib}[1]{\csdef{annbib@#1}{1}}

\annbib{NeffGrabanSchweickertMartin2020Richter}
\newcommand{\anndoi}[1]{\csdef{anndoi@#1}{1}}
\newcommand{\annbibfield}[2]{\csdef{annbibfield@#1@#2}{1}}
\anndoi{TruesdellNoll1965}
\annbibfield{XieJaviliNeffLinder2026}{note}
\annbibfield{XieJaviliNeffLinder2026}{url}
\annbib{AnssariBenamSaccomandi2025Compressible}
\annbib{AnssariBenam2021Molecular}
\annbib{KnowlesSternberg1978}
\annbib{MieheGoktepeLulei2004}
\annbib{LewisSendov2001}
\annbib{AnssariBenam2023Comprehensive}
\annbib{BucheSilbersteinGrutzik2022}
\annbib{AnssariBenamBucchiSaccomandi2021I2}
\annbib{DestradePucciSaccomandi2019}
\annbib{AnssariBenamBucchiSaccomandi2022Inflation}
\makeatletter
\def\ann@currentkey{}
\def\ann@bibitem#1{%
  \par%
  \gdef\ann@currentkey{#1}%
  \ifcsdef{newannbib@#1}{\color{newanncolor}}{\ifcsdef{annbib@#1}{\color{anncolor}}{\normalcolor}}%
  \origbibitem{#1}%
  \ifcsdef{newannbib@#1}{\color{newanncolor}}{\ifcsdef{annbib@#1}{\color{anncolor}}{\normalcolor}}%
}
\def\ann@lbibitem[#1]#2{%
  \par%
  \gdef\ann@currentkey{#2}%
  \ifcsdef{newannbib@#2}{\color{newanncolor}}{\ifcsdef{annbib@#2}{\color{anncolor}}{\normalcolor}}%
  \origbibitem[#1]{#2}%
  \ifcsdef{newannbib@#2}{\color{newanncolor}}{\ifcsdef{annbib@#2}{\color{anncolor}}{\normalcolor}}%
}
\AtBeginDocument{%
  \let\origbibitem\bibitem
  \renewcommand{\bibitem}{\@ifnextchar[\ann@lbibitem\ann@bibitem}%
  \providecommand{\DOIprefix}{doi:\ }
\providecommand{\doi}[1]{\url{https://doi.org/#1}}%
  \providecommand{\bibinfo}[2]{#2}%
  \let\origbibinfo\bibinfo
  \renewcommand{\bibinfo}[2]{%
    \ifcsdef{annbibfield@\ann@currentkey @#1}{\ann{\origbibinfo{#1}{#2}}}{\origbibinfo{#1}{#2}}%
  }%
  \let\origbiburl\url
  \renewcommand{\url}[1]{%
    \ifcsdef{annbibfield@\ann@currentkey @url}{\ann{\origbiburl{#1}}}{\origbiburl{#1}}%
  }%
  \let\origdoi\doi
  \renewcommand{\doi}[1]{%
    \ifcsdef{anndoi@\ann@currentkey}{\ann{\origdoi{#1}}}{\origdoi{#1}}%
  }%
}
\makeatother
\AtBeginEnvironment{thebibliography}{\setlength{\bibsep}{0pt}}

\newtheorem{theorem}{Theorem}
\theoremstyle{remark}
\newtheorem{remark}{Remark}

\makeatletter\newcommand{\annfigurelabel}{\def\fnum@figure{\ann{\figurename~\thefigure}}}\makeatother
\hypersetup{pdftitle={A globally defined compressible polyconvex isotropic energy satisfying true-stress-true-strain monotonicity},pdfauthor={Dongxin Bai; Yunhao Wu; Yong Li; Kai Zhang; Patrizio Neff},pdfsubject={EJMSOL-D-26-01899}}
\makeatletter
\def\ps@pprintTitle{%
  \let\@oddhead\@empty\let\@evenhead\@empty
  \def\@oddfoot{\parbox{\textwidth}{\centering\fontsize{8}{9}\selectfont
    \textcopyright\ 2026. This manuscript version is made available under the CC BY-NC-ND 4.0 license.\\
    \url{https://creativecommons.org/licenses/by-nc-nd/4.0/}}}%
  \let\@evenfoot\@oddfoot}
\makeatother
\begin{document}
\begin{frontmatter}
\title{A globally defined \ann{compressible} polyconvex isotropic energy satisfying the true-stress-true-strain monotonicity condition (\TSTSM)}
\author[tj]{Dongxin Bai}
\ead{2512255@tongji.edu.cn}
\author[tj]{Yunhao Wu}
\ead{2610189@tongji.edu.cn}
\author[tj,aau]{Yong Li\corref{cor1}}
\ead{25131@tongji.edu.cn}
\author[tj]{Kai Zhang\corref{cor2}}
\ead{kaizhang@tongji.edu.cn}
\appto{\elsauthors}{\def\authorsep{\unskip,\space\ann{and}\space}}
\author[ude]{Patrizio Neff}
\ead{patrizio.neff@uni-due.de}
\address[tj]{School of Aerospace Engineering and Applied Mechanics, Tongji University, Shanghai 200092, China}
\address[aau]{Department of Chemistry and Bioscience, Aalborg University, 9220 Aalborg East, Denmark}
\address[ude]{Chair of Nonlinear Analysis and Modelling, University of Duisburg-Essen, Thea-Leymann-Stra{\ss}e 9, 45127 Essen, Germany}
\cortext[cor1]{Corresponding author}
\cortext[cor2]{Corresponding author}
\date{}
\begin{abstract}
Polyconvexity is a standard ingredient in the variational existence theory of finite elasticity, whereas true-stress-true-strain monotonicity (\TSTSMplain) requires a positive incremental Cauchy-stress response. These two constitutive restrictions are independent, and Wollner, Holzapfel and Neff left open whether a compressible isotropic energy defined on the whole of $\mathrm{GL}^+(3)$ can satisfy both. We give an explicit affirmative answer. For every $\mu>0$ and $k>0$, the stored-energy function

\begin{equation}
W_k(F)
=\frac{\mu}{2k}
\bigl[\exp\bigl(k(\lVert F\rVert^2+3J^{-1}+J-7)\bigr)-1\bigr],
\qquad
J=\det F,
\end{equation}
is polyconvex and strictly rank-one convex. Its Cauchy-stress response satisfies \TSTSM globally if and only if $k\ge1/(8\sqrt3)$. In this regime every symmetric Cauchy stress corresponds to a unique positive-definite stretch, while the reference state is stress free with positive infinitesimal shear and bulk moduli. Stress bijectivity has the strictly smaller sharp threshold $k_{\rm B}\ann{\approx0.0083}$, defined below: at equality the stress map is a global homeomorphism with a nondifferentiable inverse, and above it the map is a global $C^\infty$ diffeomorphism. Thus, for $k_{\rm B}\le k<1/(8\sqrt3)$, the map $V\mapsto\sigma(V)$ remains globally bijective while \TSTSM fails at finite strain. In the \TSTSM regime, every prescribed inner radius of a finite plane-strain annulus with a traction-free outer wall has a unique radial equilibrium, and its inner pressure increases smoothly and strictly from zero to infinity. Comparison with the Ciarlet--Geymonat neo-Hooke (CG-NH), exponentiated Hencky, and stress-free $I_1$-exponential energies separates the roles of polyconvexity, \TSTSM, reduced convexity, and endpoint growth; for $W_k$ these ingredients yield the complete annular theorem. Under incompressible spherical inflation, global pressure inversion holds for the Demiray family with $b\ge b_*$, including $W_k|_{J=1}$ for $k\ge k_*$ and $W_{\exp}|_{J=1}$, and for the exponentiated Hencky energy when $k_{\mathrm{eH}}>3/8$; incompressible neo-Hooke instead has a limit point.
\ann{A further compressible $I_1$ energy retains these global constitutive properties and the reference moduli of $W_k$, while exhibiting two pressure turning points under incompressible spherical inflation. Thus these properties allow a descending pressure branch.}
\end{abstract}
\begin{keyword}
polyconvexity
\sep
\TSTSM
\sep
Cauchy stress
\sep
stress bijectivity
\sep
isotropic hyperelasticity
\sep
plane-strain annulus
\sep
Demiray-Fung elasticity
\sep
corotational stability postulate
\sep
logarithmic strain
\MSC[2020] 74B20 \sep 74A20 \sep 49J45
\end{keyword}
\end{frontmatter}

\section{Introduction}

In 1956, Truesdell formulated the unsolved Hauptproblem of finite elasticity as the search for ``the class of functions that may serve as strain energy densities, for perfectly elastic materials'' \citep{Truesdell1956}. Truesdell and Noll subsequently placed this question in the framework of constitutive inequalities: beyond thermodynamic compatibility, one seeks a priori conditions that exclude unreasonable elastic responses \citep{TruesdellNoll1965}. Because the second law alone does not restrict the specific form of $W$, additional constitutive restrictions must be selected. This paper addresses the resulting modern constitutive problem posed by Neff: whether a compressible isotropic hyperelastic energy defined on the whole of $\mathrm{GL}^+(3)$ can combine variational admissibility, global monotonicity and invertibility of the Cauchy-stress response, and a stable infinitesimal Hooke law.

Polyconvexity provides a standard route to the existence of minimizers by the direct method. Under suitable coercivity, growth, and boundary assumptions, a convex representation of $W$ in independent minor variables $(Y,G,\ann{J})$, evaluated at $(F,\Cof F,\det F)$, gives weak lower semicontinuity and hence variational existence \citep{Ball1976,Ciarlet1988}. Polyconvexity implies quasiconvexity \citep{Morrey1952} and rank-one convexity, the latter yielding the Legendre-Hadamard condition. Classical compressible polyconvex constructions and invariant criteria are available \citep{CiarletGeymonat1982,Steigmann2003}. These results control the dependence of the energy on the deformation gradient and its minors, but they do not decide whether the Cauchy-stress map is strictly monotone with respect to Hencky strain (the logarithmic strain).

A second line of development addresses the stress response directly. Hill's constitutive inequalities are expressed through the Zaremba-Jaumann rate of the Kirchhoff stress and, for hyperelasticity, relate monotonicity of the Kirchhoff-stress/Hencky-strain map to convexity of the energy as a function of $\log V$ \citep{Hill1968,Hill1970}. Leblond subsequently imposed the analogous rate inequality on Cauchy stress and introduced the true-stress-true-strain monotonicity condition now denoted \TSTSM \citep{Leblond1992}. Let $X=\log V$ be the Eulerian Hencky strain and set $\widehat\sigma(X)=\sigma(e^X)$. The condition requires
\begin{equation}
\label{eq:tsts-local}
\bigl\langle \Diff_X\widehat\sigma(X).Z,\,Z\bigr\rangle>0
\qquad\forall X\in\operatorname{Sym}(3)
\qquad\forall Z\in\operatorname{Sym}(3)\setminus\{0\}.
\end{equation}
This condition requires the symmetric part of the stress derivative to be positive definite at every Hencky strain. For a continuously differentiable response, integration along $X(t)=(1-t)X_2+tX_1$ in the convex space $\operatorname{Sym}(3)$ gives strict Hilbert monotonicity, denoted TSTS-M\textsuperscript{+} \citep[Eqs.~(1.13)--(1.14)]{Wollner2026}:
\begin{equation*}
\bigl\langle\widehat\sigma(X_1)-\widehat\sigma(X_2),\,X_1-X_2\bigr\rangle>0
\qquad\forall X_1\ne X_2\in\operatorname{Sym}(3).
\end{equation*}
Thus global \TSTSM implies TSTS-M\textsuperscript{+}; the converse does not hold in general. The differential condition is tested on each tangent, whereas its global validity constrains the stress difference between any two finite strain states. The finite-increment inequality gives global injectivity of $V\mapsto\sigma(V)$; surjectivity requires a separate argument. Theorem~3.25 of \citet[arXiv:2608.07087v1, Section~3.4, p.~18]{Voss2026} relates this finite-increment monotonicity of isotropic tensor functions to that of their principal-value representations.

The corotational stability postulate (CSP) requires
\begin{equation*}
\bigl\langle \tfrac{\mathrm{D}^{\circ}}{\mathrm{D}t}\sigma,\,D\bigr\rangle>0
\qquad\text{for every }D\neq0,
\end{equation*}
where $D=\operatorname{sym}(\dot F F^{-1})$ is the rate-of-deformation tensor and $\mathrm D^{\circ}/\mathrm Dt$ denotes a specified corotational rate. For isotropic Cauchy elasticity, CSP for the Zaremba--Jaumann rate is equivalent to \TSTSM \citep[Section~3, Remark~A.8]{DAgostino2025}. The \ann{equivalent} statement for other suitable corotational rates is treated separately in \citet{VossMartinNeffInPrep}. For a perfect fluid, the condition reduces to the classical requirement that pressure increase with density.

The simplest explicit constitutive law satisfying \TSTSM is Hencky's 1928 Cauchy-elastic relation \citep{Hencky1928},
\begin{equation}
\sigma(V)=\widehat\sigma(\log V)=2\mu\log V+\lambda_{\mathrm L,0}\operatorname{tr}(\log V)\Id,
\qquad \mu>0,\quad 2\mu+3\lambda_{\mathrm L,0}>0.
\end{equation}
Its derivative with respect to \ann{$\log V$} is the positive-definite isotropic Hooke tensor. This law is not hyperelastic and is not globally Legendre-Hadamard elliptic.

At the level of local constitutive restrictions, Leblond showed that his Cauchy-stress rate inequality
\begin{equation*}
\bigl\langle \tfrac{\mathrm{D}^{\mathrm{ZJ}}}{\mathrm{D}t}\sigma,\,D\bigr\rangle>0
\qquad\text{for every }D\neq 0
\end{equation*}
and polyconvexity are independent in general \citep{Leblond1992}. To the best of our knowledge, the earliest explicit hyperelastic energy known to satisfy \TSTSM globally is the full-logarithmic-norm exponentiated Hencky energy $W_{\mathrm{eH}}(F)=\frac{\mu}{k_{\mathrm{eH}}}\exp(k_{\mathrm{eH}}\lVert\log V\rVert^2)$, with $k_{\mathrm{eH}}>3/8$, introduced by Neff, Ghiba and Lankeit \citep{Neff2015ExpHencky}. Martin, Ghiba and Neff subsequently proved that no strictly increasing outer function of $\lVert\log V\rVert^2$ yields a rank-one-convex energy on all of $\mathrm{GL}^+(3)$ \citep{Martin2018}. Although $W_{\mathrm{eH}}$ is not globally rank-one convex, its Cauchy shear stress remains monotone increasing in simple shear. Conversely, the function $F\mapsto\exp(\lVert F\rVert^2-2\log\det F)$ is polyconvex but fails \TSTSM. Taken together, these examples show that an exponential envelope alone does not reconcile polyconvexity with \TSTSM.

Studies surrounding the 2024 IUTAM Symposium established the links between \TSTSM, corotational stability, and local stress invertibility \citep{Neff2024}. Neff et al. studied corotational stability in the hypoelastic rate formulation, established the corresponding equivalence for the logarithmic rate, and identified as a major open question the search for an objective isotropic energy that is polyconvex, or at least Legendre-Hadamard elliptic, while satisfying the corotational stability postulate globally \citep{Neff2025}. A companion study derived positive incremental Cauchy moduli along homogeneous diagonal deformations and local invertibility of the Cauchy stress-stretch relation \citep{Neff2025CSP}. A later rate-form analysis, alluded to above, showed how the same positivity enters the induced fourth-order tangent stiffness of the spatial equilibrium system and exhibited a uniformly positive-definite Cauchy-elastic example \citep{Neff2026RateForm}, thereby establishing the analytical role of the corotational stability postulate in rate-form equilibrium. Collectively, these works clarified the differential condition and its local analytical consequences; the unresolved step was a finite-valued polyconvex energy with global stress inversion.

In their IUTAM report, Neff and coauthors formulated this compatibility problem in terms of four requirements; the report was subsequently circulated on ResearchGate \citep{Neff2024}. The three-dimensional challenge asks for a compressible isotropic energy $W$ that simultaneously satisfies: (a) $W$ is defined on all of $\mathrm{GL}^+(3)$ and is polyconvex, or at least rank-one convex; (b) its Cauchy stress satisfies \TSTSM globally with respect to Hencky strain; (c) $V\mapsto\sigma(V)$ is a bijection from $\operatorname{Sym}^{++}(3)$ onto $\operatorname{Sym}(3)$; and (d) the reference-state linearization is an isotropic Hooke law with positive shear and bulk moduli. A nonexistence proof is also admitted. Since isotropy makes the deformation gradient $F$ and its right composition $FQ$ with a rotation $Q\in\mathrm{SO}(3)$ produce the same Cauchy stress, requirement (c) is stated in the positive-definite stretch variable $V$.

Subsequent investigations \citep{Ghiba2026} clarified why local diagnostics and separate stability conditions do not resolve this challenge. Ghiba et al. numerically compared local invertibility and local monotonicity for a collection of named isotropic energies and discussed the distinction between local and global inversion. A nonsingular stress tangent gives only local invertibility; \TSTSM integrates on the convex space $\operatorname{Sym}(3)$ to strict Hilbert monotonicity and hence global injectivity, while bijectivity additionally requires every symmetric Cauchy stress to be attained. In a complementary direction, Wollner, Holzapfel and Neff restated the IUTAM problem as Challenge (i) and made the independence of polyconvexity and \TSTSM concrete through explicit examples \citep{Wollner2026}. Their three constructions satisfying both conditions are finite only on domains defined by $\lVert F\rVert^\alpha<\beta$, $\lVert\Cof F\rVert^\alpha<\beta$, or $\log^2\det F<\beta$, and take the value $+\infty$ outside those domains. The full-domain finite-valued problem therefore remained open in that reference.

\ann{Existing studies have investigated the compatibility of polyconvexity and stress monotonicity under different deformation constraints and on different energy domains.} Under incompressibility, \citet{WollnerKlein2026Concurrent} proved \ann{that a broad class of energies} covered by Ball-type sufficient parameterizations \ann{can simultaneously satisfy} polyconvexity and the corresponding TSTS monotonicity condition. \ann{This connection also depends on dimension: in two dimensions,} every differentiable, objective and isotropic polyconvex energy satisfies the weak Hill inequality \citep{GhibaWollnerNeff2027}; in three dimensions, polyconvexity alone \ann{is insufficient to ensure} TSTS monotonicity \citep{Klein2026}. \ann{For compressible deformations, the limiting-chain-extensibility models of \citet{AnssariBenamSaccomandi2025Compressible} provide another basis for such constructions: their independently selectable volumetric contribution offers freedom to investigate the constitutive requirements above. In the limiting-chain parameter regime, the energy diverges at a finite extensibility boundary to describe the extensibility limit of molecular chains. The constitutive conditions for these models are therefore considered on the domain where the energy is finite. The present paper further examines the compatibility of these requirements on the whole of $\mathrm{GL}^+(3)$, requiring the energy to be smooth and finite-valued at every orientation-preserving invertible deformation.}

The present paper gives an explicit affirmative answer to Challenge (i). For every $\mu>0$ and $k>0$, the energy $W_k$ defined in Section 2 is finite-valued, smooth, polyconvex, and strictly rank-one convex on the whole of $\mathrm{GL}^+(3)$; its reference state is stress-free and its infinitesimal moduli are positive. If $k\ge k_*:=1/(8\sqrt3)$, its Cauchy-stress response satisfies \TSTSM globally and $V\mapsto\sigma(V)$ is bijective. Thus, to the best of our knowledge, $W_k$ is the first explicit finite-valued smooth energy on all of $\mathrm{GL}^+(3)$ that satisfies the four requirements simultaneously. The same family separates two sharp thresholds: $k_*$ controls \TSTSM, whereas stress bijectivity is already attained at the strictly smaller value $k_{\rm B}$.

\ann{We \ann{will} call a smooth, objective and isotropic compressible hyperelastic model \emph{hyper-Hooke} if its energy is finite and polyconvex on all of $\mathrm{GL}^+(3)$ and its Cauchy stress satisfies global \TSTSM{} with respect to logarithmic strain. The name is chosen by analogy with classical neo-Hooke models, whose energies contain a term linear in $I_1=\lVert F\rVert^2$. The definition allows $I_2$ dependence and imposes no condition on the Kirchhoff stress. As in positive-definite linear elasticity, these models have rank-one convex energies and strictly monotone stress responses, with the finite-strain monotonicity expressed for the pair $(\sigma,\log V)$. The energy $W_k$ gives an explicit hyper-Hooke model for $k\ge k_*$.}

For the traction-free plane-strain annulus and $k\ge k_*$, Appendix~B proves that $a\mapsto p_i(a)$ is a bijection from $[A,\infty)$ onto $[0,\infty)$. Here $p_i(a)$ is the internal pressure required to maintain the current cavity radius $a$, with a traction-free outer boundary; the subscript $i$ denotes the inner wall. With reference inner radius $A$ and $a=r(A)$, the pressure is $p_i(a)=-\sigma_r(A)$, using the reference coordinate to label the stress field. Appendix B compares $W_k$ with the polyconvex Ciarlet--Geymonat neo-Hooke (CG-NH) law, the \TSTSM exponentiated Hencky law, and a polyconvex stress-free $I_1$ exponential that fails \TSTSM. For $W_k$, radial existence and uniqueness, strict pressure increase, and unbounded growth combine into the global bijection. For incompressible spherical inflation, Appendix~C proves global pressure inversion for the Demiray family with $b\ge b_*$, including $W_k|_{J=1}$ for $k\ge k_*$ and $W_{\exp}|_{J=1}$; the exponentiated Hencky energy has the same inversion for $k_{\mathrm{eH}}>3/8$ by a separate closed formula. Thus constitutive laws with different three-dimensional polyconvexity and \TSTSM properties share the same spherical-membrane pressure inversion.
\ann{The absence of spherical pressure turning points is specific to this family in the stated range. Appendix C identifies the smaller Demiray threshold $b_c\approx 0.067<k_*$ and distinguishes it from the three-dimensional TSTS condition. Other generalized neo-Hookean energies can exhibit both a limit point and a second ascending branch \citep{AnssariBenamBucchiSaccomandi2022Inflation}. \hyperref[sec:own-i1-balloon]{C.5} derives the pressure-growth criterion. \hyperref[sec:raw-invariant-balloons]{C.6} then gives a compressible $I_1$ energy with the same reference moduli as $W_k$. It retains global polyconvexity, strict ellipticity, \TSTSM and stress bijectivity while admitting spherical pressure decrease. Section~\ref{sec:scope} examines these structural responses together with the material limitations of the energy.}

The core constitutive conclusions are as follows.

\begin{theorem}\label{thm:1} Let $\mu>0$ and $k>0$. The energy $W_k$ introduced in Section 2 is smooth ($C^\infty$) on $\mathrm{GL}^+(3)$, objective and isotropic, polyconvex and strictly rank-one convex. The reference state is stress free, with infinitesimal moduli $\mu$ and $\kappa_0=8\mu/3$. If $k\ge k_*:=1/(8\sqrt3)$, its Cauchy-stress response satisfies \TSTSM globally and the map $V\mapsto\sigma(V)$ is a bijection from $\operatorname{Sym}^{++}(3)$ onto $\operatorname{Sym}(3)$.

\end{theorem}

{\color{anncolor}
\begin{remark}[\ann{Compressible} Hill monotonicity]\label{rem:hill}
For every $k>0$, the Kirchhoff stress $\tau_k=J\sigma_k$ satisfies global strict Hill monotonicity with respect to $\log V$, and $V\mapsto\tau_k(V)$ is a global $C^\infty$ diffeomorphism from $\operatorname{Sym}^{++}(3)$ onto $\operatorname{Sym}(3)$. For this family, the two properties hold even when $0<k<k_*$ and global Cauchy-stress \TSTSM{} fails. For an incompressible response, $J=1$ and $\tau=\sigma$, so Hill and TSTS monotonicity coincide on the admissible trace-free logarithmic strains \ann{\citep{WollnerKlein2026Concurrent}}. The indeterminate hydrostatic pressure does not contribute to the corresponding stress-strain pairings.
\end{remark}
}

\begin{theorem}\label{thm:2} If $0<k<k_*$, then $W_k$ fails \TSTSM at a finite deformation.

\end{theorem}

\ann{Theorems~\ref{thm:1} and~\ref{thm:2} show that $W_k$ is a hyper-Hooke model if and only if $k\ge k_*$.}

\begin{theorem}\label{thm:3} Let $\ell_{\rm B}$ be the unique positive root of

\begin{equation}
r(\ell):=4\ell^{11}+4\ell^{10}-81\ell^6-102\ell^5-81=0,
\end{equation}
and set

\begin{equation}
k_{\rm B}
=\frac{2\ell_{\rm B}^3(\ell_{\rm B}^5-9)}{3(\ell_{\rm B}^6+2\ell_{\rm B}^5-3)^2}
\ann{\approx 0.0083}.
\end{equation}
If $0<k<k_{\rm B}$, the map $V\mapsto\sigma(V)$ is not injective. If $k=k_{\rm B}$ it is a global homeomorphism, but the inverse is not differentiable at the critical stress. If $k>k_{\rm B}$ it is a global $C^\infty$ diffeomorphism. In particular, for every $k_{\rm B}\le k<k_*$ the map $V\mapsto\sigma(V)$ is globally bijective while \TSTSM fails.

\end{theorem}

The subscript B labels the bijectivity threshold. The proofs are organized around the two thresholds. Section 2 derives the energy and its Cauchy stress. Section 3 \ann{establishes} polyconvexity, the reference-state linearization, and global \TSTSM for $k\ge k_*$, and then combines local invertibility with properness of the stress map, established through its unbounded growth along escaping stretch sequences, to obtain global bijectivity in that regime. \ann{Section 3.2 also briefly discusses strict Hill monotonicity and Kirchhoff-stress bijectivity for every $k>0$.}
\ann{\ann{Section~\ref{sec:proper}} records an $I_2$ analogue of $W_k$, with a parallel argument.} The isotropic spectral derivative decomposition reduces \TSTSM to positive definiteness of the symmetric principal-stress tangent together with positive noncoaxial moduli \citep{Leblond1992,Wollner2026}. Section 4 constructs a finite-deformation witness showing that $k_*$ cannot be lowered. Section 5 instead studies the unsymmetrized stress tangent and its hydrostatic singularity, determines the bijectivity threshold $k_{\rm B}$, and proves Theorem 3; the distinction between global inversion and local tangent conditions follows \citet{Ghiba2026,Silhavy1997}. Appendix A records closed-form Cauchy-stress and incremental responses in four homogeneous deformation modes. Appendix~B proves the complete annular pressure-inversion theorem and compares $W_k$ with the Ciarlet--Geymonat neo-Hooke, exponentiated Hencky, and stress-free $I_1$-exponential energies at radial-stress inversion, reduced convexity, global shooting, pressure monotonicity, and endpoint growth. Appendix~C proves spherical-membrane pressure inversion for the Demiray family with $b\ge b_*$ and derives the limit point of incompressible neo-Hooke. \ann{Figure 1(b) records the global Kirchhoff-stress bijectivity and strict Hill monotonicity for all $k>0$.} Figure 1 places both sharp thresholds and their failure mechanisms in one parameter diagram: the hydrostatic folds in panel \ann{(c)} produce noninjectivity below $k_{\rm B}$, while the finite-distortion direction in panel \ann{(d)} proves the loss of \TSTSM below $k_*$.

\begin{figure}[ht]
\centering
\includegraphics[width=\textwidth]{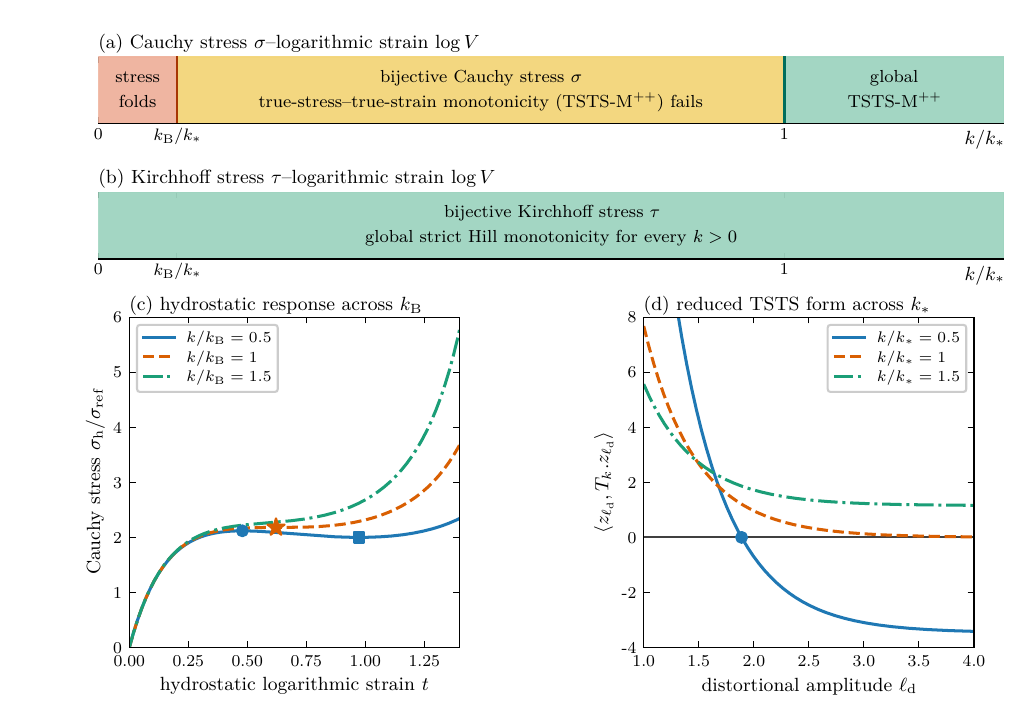}
\caption{\ann{Cauchy-stress thresholds and Kirchhoff-stress monotonicity for $W_k$. The stress-strain pairs in (a) and (b) are $(\sigma,\log V)$ and $(\tau,\log V)$, respectively, with $\tau=J\sigma$. Panels (a) and (b) each display the same parameter axis $k/k_*$, with $k>0$. In (a), hydrostatic Cauchy-stress folds occur for $0<k<k_{\rm B}$. At $k=k_{\rm B}$, the stress map is a global homeomorphism with a nondifferentiable inverse, and it is a global smooth diffeomorphism for $k>k_{\rm B}$. Global \TSTSM{} holds exactly when $k\ge k_*$. The continuous green band in (b) indicates bijectivity of $V\mapsto\tau_k(V)$ from $\operatorname{Sym}^{++}(3)$ onto $\operatorname{Sym}(3)$ and strict Hill monotonicity for every $k>0$. Each prescribed symmetric Kirchhoff stress corresponds to a unique positive-definite stretch, and the inverse map is smooth. In (c), the blue, orange and green curves correspond to $k/k_{\rm B}=0.5,1,1.5$. The circle and square mark the hydrostatic stress extrema below $k_{\rm B}$, and the star marks the stationary inflection at $k_{\rm B}$. Panel (d) shows the reduced Cauchy-stress TSTS form $\langle z_{\ell_{\mathrm d}},T_k.z_{\ell_{\mathrm d}}\rangle=A_k^{-1}\langle\Diff_xs_k(x_{\ell_{\mathrm d}}).z_{\ell_{\mathrm d}},z_{\ell_{\mathrm d}}\rangle$ evaluated at the fixed-volume states used in Section 4. The test increments allow volume changes. The curves correspond to $k/k_*=0.5,1,1.5$. Below $k_*$, the form becomes negative at finite distortion. At $k_*$, it remains positive at every finite distortion and tends to zero as $\ell_{\mathrm d}\to\infty$. Above $k_*$, its limit is positive. The green parameter regions in (a) and (b) continue beyond the displayed window.}}
\label{fig:1}
\end{figure}

\section{The stored-energy function and its Cauchy stress}

Let $F\in\mathrm{GL}^+(3)$ and $J=\det F$. \ann{The squared Frobenius norm is $\lVert F\rVert^2=\operatorname{tr}(F^{\mathsf T}F)=\sum_{i,j=1}^3F_{ij}^2$. We use $I_1=\lVert F\rVert^2$ and $I_2=\lVert\operatorname{Cof} F\rVert^2$ for the first two principal invariants of the Cauchy-Green tensors.} For $\mu>0$ and $k>0$ define

\begin{equation}
P(F)=\lVert F\rVert^2+3J^{-1}+J,
\qquad
W_k(F)
=\frac{\mu}{2k}
\bigl[\exp\bigl(k(P(F)-7)\bigr)-1\bigr].
\end{equation}
Here $\mu>0$ is the infinitesimal shear modulus, $\sigma_{\mathrm{ref}}=\mu/2$ is the stress scale used in the figures, and $k$ is a dimensionless finite-strain stiffening parameter. The term $\lVert F\rVert^2$ measures the overall stretch level, $3J^{-1}$ provides a barrier against vanishing volume, and $J$ contributes to the resistance against volumetric expansion. The coefficients are chosen so that $2+1-3=0$ makes the reference state stress free, while $7=3+3+1$ \ann{sets its reference energy to zero}; the minimum curvature of $3e^{-s}+e^s$ will determine $k_*$.

\ann{The scalar $P$ is the inner function used to construct the energy; its zero-reference form is $q(F):=P(F)-7$. For a proper rotation $R\in\mathrm{SO}(3)$, $\lVert R\rVert^2=3$ and $J=1$, hence $P(R)=7$ and $W_k(R)=0$. Subtracting $7$ sets the reference energy to zero.}

\ann{Expressed in the principal stretches, $q(F)$ takes the form}
\begin{equation*}
\ann{q(F)=\lambda_1^2+\lambda_2^2+\lambda_3^2+
\frac{3}{\lambda_1\lambda_2\lambda_3}+\lambda_1\lambda_2\lambda_3-7.}
\end{equation*}
\ann{\ann{This function has its unique global minimum at} $\lambda_1=\lambda_2=\lambda_3=1$. Consequently, $q(F)\ge0$ on $\mathrm{GL}^+(3)$, with $q(F)=0$ if and only if $F\in\mathrm{SO}(3)$, as proved in Section 3.2. Since the outer exponential is strictly increasing, $W_k(F)=0$ precisely for proper rotations.} \ann{Thus $q(F)$ is an objective, isotropic scalar strain measure.}

Write $B=FF^{\mathsf T}$, $V=\sqrt B$ and $X=\log V$, where $V$ is the unique positive-definite square root of $B$. \ann{If $\lambda_i$ are the principal stretches, then $x_i=\log\lambda_i$ are the principal values of $X$. The tensor variable $X\in\operatorname{Sym}(3)$ is represented by the vector $x=(x_1,x_2,x_3)\in\mathbb R^3$ in a principal basis.} \ann{Set} $s=x_1+x_2+x_3=\log J$, let $\mathbf1=(1,1,1)^{\mathsf T}$, and let $\Id$ denote the identity \ann{in the relevant} tensor \ann{or matrix space}. \ann{We retain the tensor symbols $F,B,V,X,\sigma,\tau$. The symbols $x,g,s_k$ denote vectors, $H,\Delta$ below denote matrices, and $P,p,J,s$ are scalars. The matrix $\Delta$ is the Hessian of $I_1=\sum_i e^{2x_i}$ with respect to $x$, namely $\Delta=\operatorname{diag}(4e^{2x_1},4e^{2x_2},4e^{2x_3})$. In the Introduction, $D=\operatorname{sym}(\dot F F^{-1})$ is the rate-of-deformation tensor, i.e.\ the symmetric part of the spatial velocity gradient, and the dot denotes the material time derivative. The upright symbol $\mathrm D$ denotes differentiation.} The principal representation \ann{satisfies $p(x)=P(\operatorname{diag}(e^{x_1},e^{x_2},e^{x_3}))=P(F)$. For an increment $z\in\mathbb R^3$, scalar derivatives are represented by vectors or tensors using the Euclidean or Frobenius inner product: for example, $\left.\mathrm d p(x+tz)/\mathrm dt\right|_{t=0}=\langle\Diff_xp(x),z\rangle$ and $\left.\mathrm d^2p(x+tz)/\mathrm dt^2\right|_{t=0}=\langle\Diff_x^2p(x).z,z\rangle$. A tensor-valued stress derivative is a linear map, with action $\Diff_X\widehat\sigma(X).Z$. All derivatives of $p$ are with respect to the principal logarithmic stretches: $g_i=\partial p/\partial x_i$ and $H_{ij}=\partial^2p/(\partial x_i\partial x_j)$. Explicitly,}

\begin{equation}
p(x)=\sum_{i=1}^3\exp(2x_i)+3\exp(-s)+\exp(s),
\qquad
g=\Diff_xp,
\qquad
H=\Diff_x^2p.
\end{equation}
Componentwise,

\begin{equation}
g_i=2\exp(2x_i)+\exp(s)-3\exp(-s),
\end{equation}

\begin{equation}
H_{ij}=4\exp(2x_i)\delta_{ij}+(3\exp(-s)+\exp(s)).
\end{equation}
\ann{Here $i,j\in\{1,2,3\}$, and no summation over $i$ or $j$ is intended in this component formula.}
Hence

\begin{equation}
H=\Delta+\psi''(s)\,\mathbf1\otimes\mathbf1,
\end{equation}
where

\begin{equation}
\Delta=\operatorname{diag}(4\exp(2x_1),4\exp(2x_2),4\exp(2x_3)),
\qquad
\psi(s)=3\exp(-s)+\exp(s).
\end{equation}
Since each diagonal entry $4\exp(2x_i)$ is positive, $\Delta$ is positive definite for every $x\in\mathbb R^3$. For the present isotropic hyperelastic law, Richter's formula gives the work-conjugacy relation $\tau_k=\Diff_X\widehat W_k(X)$, where $\widehat W_k(X):=W_k(e^X)$ \citep[Section~3, Eqs.~(3.8*)--(3.9)]{Richter1948}. The Cauchy stress is $\sigma_k=e^{-\operatorname{tr}X}\Diff_X\widehat W_k(X)=J^{-1}\tau_k$, namely

\begin{equation}
\sigma_k(F)
=\frac{\mu}{2}\exp\bigl(k(P(F)-7)\bigr)
\bigl[2J^{-1}B+(1-3J^{-2})\Id\bigr].
\end{equation}
Its principal stress vector is

\begin{equation}
s_k(x)=A_k(x)g(x),
\qquad
A_k(x):=\frac{\mu}{2}\exp\bigl(-s+k(p(x)-7)\bigr)>0.
\end{equation}
Here $g$ gives the principal response of the inner energy, $\exp(-s)$ converts Kirchhoff stress to Cauchy stress, and $\exp(k(p-7))$ introduces the finite-strain stiffening controlled by $k$. The two thresholds below are set jointly by the strain curvature of the inner energy, the volumetric Cauchy-stress conversion, and this exponential stiffening. The Cauchy stress in four homogeneous deformation modes is given in Appendix A.

\section{Proof of Theorem 1: polyconvexity, linearization, and \TSTSM}

\subsection{Polyconvexity and rank-one convexity}

To define a convex representative, \ann{distinguish the physical deformation gradient $F\in\mathrm{GL}^+(3)$ from the independent minor variables $(Y,G,J)\in\mathbb R^{3\times3}\times\mathbb R^{3\times3}\times\mathbb R$. Here $J$ is an independent scalar argument. It is set to $\det F$ only when the representative is evaluated on the physical minors.} Set
\begin{equation}
\mathcal P(Y,G,\ann{J})=
\begin{cases}
\lVert Y\rVert^2+3\ann{J}^{-1}+\ann{J},&\ann{J}>0,\\
+\infty,&\ann{J}\le0.
\end{cases}
\end{equation}
The first summand is convex in $Y$ and independent of $(G,\ann{J})$. The function $\ann{J}\mapsto3\ann{J}^{-1}+\ann{J}$ is strictly convex on $(0,\infty)$; its extension by $+\infty$ is closed and convex. Thus $\mathcal P$ is convex in the independent minors \citep{Ball1976,Ciarlet1988}, and
\begin{equation*}
P(F)=\mathcal P(F,\Cof F,\det F),\qquad q(F)=P(F)-7.
\end{equation*}
The representative uses the first and third minors and is independent of $G$. The logarithmic-strain variable $X$ used in the stress analysis is a separate coordinate.

The outer function $t\mapsto k^{-1}(\exp(k(t-7))-1)$ is convex and increasing, so $\mu/2$ times its composition with $\mathcal P$ remains a convex function of the minors. This is polyconvexity of $W_k$ for every $k>0$. Let $L=\xi_F\otimes\eta_F$ with $\xi_F,\eta_F\in\mathbb R^3\setminus\{0\}$. Along an admissible rank-one line $F+tL$, the determinant $J(t)$ is affine and

\begin{equation}
\frac{\mathrm d^2}{\mathrm dt^2}P(F+tL)
=2\lVert L\rVert^2+\frac{6(J'(t))^2}{J(t)^3}>0.
\end{equation}
The outer derivative is strictly positive on the physical range, so the composition remains strictly convex in $t$. Hence $W_k$ is strictly rank-one convex along every nonzero rank-one deformation mode.
Because $P$ depends only on the singular values of $F$ and on $J$, the energy is objective and isotropic \citep{TruesdellNoll1965}. Every term is smooth for $J>0$, so $W_k$ is smooth ($C^\infty$) on $\mathrm{GL}^+(3)$.

To verify directly that this explicit construction also satisfies the invariant criterion used in the open problem stated in the Introduction, we now perform an independent check in the framework of \citet{Wollner2026}. Let

\begin{equation}
K_1=\lVert F\rVert,
\qquad
K_2=\lVert\Cof F\rVert,
\qquad
K_3=J=\det F,
\end{equation}
and write $W_k(F)=\Psi_k(K_1,K_3)$, where

\begin{equation}
\Psi_k(K_1,K_3)
=\frac{\mu}{2k}
\left[
\exp\!\left(k\left(K_1^2+\frac3{K_3}+K_3-7\right)\right)-1
\right].
\end{equation}
The inner function $K_1^2+3K_3^{-1}+K_3$ has Hessian $\operatorname{diag}(2,6K_3^{-3})\succ0$. Composition with the convex increasing exponential shows that $\Psi_k$ is convex in $(K_1,K_3)$ and increasing in $K_1>0$. Regarded as a function of $(K_1,K_2,K_3)$, it is independent of, and hence nondecreasing in, $K_2$. Thus $W_k$ satisfies the sufficient polyconvexity conditions of Theorem 3.1 in \citet{Wollner2026} for every $k>0$.

\subsection{\ann{Reference state and Hill monotonicity}}\label{sec:hill}

At $x=0$ one has $p(0)=7$ and $g(0)=0$. The Hessian $H$ is positive definite for every $x\in\mathbb R^3$, so $x=0$ is the unique minimizer of $p$, or equivalently $V=\Id$ is the unique minimizing stretch. By objectivity, the minimizing deformation gradients are all rotations $F\in\mathrm{SO}(3)$. Thus $W_{k}(\Id)=0$ and the reference Cauchy stress vanishes for every $k>0$.

{\color{anncolor}
Richter's formula gives $\tau_k=\Diff_{\log V}\widehat W_k(\log V)$, where $\widehat W_k(X)=W_k(e^X)$ (\citealp{Richter1948}; \citealp[Eq.~(17)]{NeffGrabanSchweickertMartin2020Richter}). For every $k>0$, the second derivative of $\widehat W_k$ is positive definite on $\operatorname{Sym}(3)$. Hence $\tau_k$ is strictly monotone in $\log V$, and $V\mapsto\tau_k(V)$ is injective. \ann{The superlinear growth of $\widehat W_k$ ensures that every symmetric Kirchhoff stress is attained.} Figure 1 compares the Kirchhoff and Cauchy stress responses.
\par}

\subsection{Reference linearization}\label{sec:reference}

The derivative of $s_k$ at the origin receives no contribution from the second derivative of the outer function, because that term is multiplied by $g(0)=0$. Thus the exponential outer function controlled by $k$ does not alter the reference tangent and affects the constitutive response only at finite strain. One obtains

\begin{equation}
\Diff_x s_k\big|_{x=0}
=\frac{\mu}{2}H(0)
=2\mu\,\Id+2\mu\,\mathbf1\otimes\mathbf1.
\label{eq:reference-tangent}
\end{equation}
\ann{Equation~\eqref{eq:reference-tangent} is written in the principal logarithmic stretches $x_i=\log\lambda_i$. Here $\mathbf1=(1,1,1)^{\mathsf T}$, and $\ann{(\mathbf1\otimes\mathbf1).z}=\langle\mathbf1,z\rangle\mathbf1$ for $z\in\mathbb R^3$. For the diagonal increment $Z=\operatorname{diag}z$, $\langle\mathbf1,z\rangle=\operatorname{tr}Z$.}
Comparison with the isotropic Hooke law $2\mu\Id+\lambda_{\mathrm L,0}\,\mathbf1\otimes\mathbf1$ \citep{TruesdellNoll1965,Ogden1970} yields

\begin{equation}
\lambda_{\mathrm L,0}=2\mu,
\qquad
\kappa_0=\lambda_{\mathrm L,0}+\frac{2\mu}{3}=\frac{8\mu}{3},
\qquad
\nu_0=\frac{\lambda_{\mathrm L,0}}{2(\lambda_{\mathrm L,0}+\mu)}=\frac13.
\end{equation}
Here $\lambda_{\mathrm L,0}$ is the first Lam\'e constant, $\kappa_0$ is the reference bulk modulus, and $\nu_0$ is the reference Poisson ratio. These moduli are independent of $k$ and strictly positive. Thus all members of the family have the same stable infinitesimal response; the distinct global thresholds in Theorems 1-3 are genuinely finite-strain phenomena.

\ann{The original family has the fixed ratio $\kappa_0/\mu_0=8/3$, corresponding to $\nu_0=1/3$. Here $\mu_0=\mu$ and $\kappa_0=8\mu/3$ denote the reference shear and bulk moduli. The bulk and shear moduli therefore cannot be prescribed independently within this family.}

\ann{The reference Poisson ratio can be raised by an additional volumetric term. Let $h\in C^\infty(0,\infty)$ be convex with $h(1)=h'(1)=0$, so that $J=1$ minimizes $h$, and set}
\begin{equation*}
\widetilde W_{k,h}(F)=W_k(F)+h(J),
\qquad
\widetilde\sigma_{k,h}(F)=\sigma_k(F)+h'(J)\Id.
\end{equation*}
\ann{For $k\ge k_*$, the extended energy \ann{retains the following properties of $W_k$}: it is polyconvex and strictly rank-one convex, stress free at the reference state and globally \TSTSM{}, and its Cauchy-stress map is a global $C^\infty$ diffeomorphism. The added representative $h(\ann{J})$ is convex in the independent scalar minor $\ann{J}$, and $J=\det F$ is affine along rank-one lines, so the arguments of Section 3.1 apply. With $J=e^{\operatorname{tr}X}$, $s=\operatorname{tr}X$ and $\eta:=h''(1)\ge0$, the added hydrostatic stress contributes}
\begin{equation*}
\bigl\langle\Diff_X[h'(J)\Id].Z,Z\bigr\rangle
=h''(J)\,J\,(\operatorname{tr}Z)^2\ge0,
\qquad
\langle h'(J)\Id,X\rangle=h'(e^{s})\,s\ge0.
\end{equation*}
\ann{The first term preserves the strict \TSTSM inequality of Section \ann{3.4}. The second keeps the lower bound for the stress pairing used in Section \ann{3.5}, so properness and global inversion follow as for $W_k$. The reference moduli become}
\begin{equation*}
\mu_0=\mu,\qquad
\lambda_{\mathrm L,0}=2\mu+\eta,\qquad
\kappa_0=\frac{8\mu}{3}+\eta,\qquad
\nu_0=\frac{2\mu+\eta}{2(3\mu+\eta)}.
\end{equation*}
\ann{Hence $1/3\le\nu_0<1/2$. The simplest choice $h(J)=\eta(J-1)^2/2$ attains every $\nu_0$ in this interval with $\eta/\mu=(6\nu_0-2)/(1-2\nu_0)$. Increasing $\eta$ raises the bulk modulus at fixed shear modulus and moves the response continuously towards slight compressibility, with exact incompressibility as the constrained limit. An exponentiated Hencky volumetric term provides another convex choice \citep{Neff2015ExpHencky}:
\begin{equation*}
h(J)=\frac{\eta}{2\ann{\widetilde k}}\left[e^{\ann{\widetilde k}(\log J)^2}-1\right],\qquad
\eta\ge0,\quad\ann{\widetilde k}\ge\frac18.
\end{equation*}
Here $\eta$ has the units of stress and is the reference bulk-modulus increment, while $\ann{\widetilde k}$ is a dimensionless exponent parameter, distinct from $k$. This choice has $h(1)=h'(1)=0$ and $h''(1)=\eta$. Indeed,
\begin{equation*}
h''(J)=\frac{\eta e^{\ann{\widetilde k}(\log J)^2}}{J^2}
\left[2\ann{\widetilde k}\left(\log J-\frac1{4\ann{\widetilde k}}\right)^2+1-\frac1{8\ann{\widetilde k}}\right]\ge0,
\end{equation*}
\ann{The bound $\widetilde k\ge1/8$ ensures convexity in $J$, including equality.} Both choices therefore satisfy the stated convexity assumptions. Since $h$ vanishes at $J=1$, the incompressible pressure responses in Appendix C are unchanged. Our subsequent work treats a wider range of reference Poisson ratios, for which the shear and bulk moduli must be prescribed independently.}

\subsection{True-stress-true-strain monotonicity}

\TSTSM asks whether the incremental Cauchy stress has a positive pairing with every nonzero logarithmic-strain increment. In what follows, $z\in\mathbb R^3$ denotes the increment of the three principal Hencky strains. Differentiating the principal stress vector gives the generally unsymmetric principal tangent

\begin{equation}
\Diff_x s_k
=A_k
\bigl(H+k\,g\otimes g-g\otimes\mathbf1\bigr).
\end{equation}
\ann{Here $H=\mathrm D_x^2p$ is the Hessian of the inner function $p$ introduced in Section 2.} $k g\otimes g$ is the nonnegative stiffening contribution from the exponential outer function, and $-g\otimes\mathbf1$ results from differentiating $J^{-1}=\exp(-s)$ when converting Kirchhoff stress to Cauchy stress.

After division by the positive prefactor $A_k$, the sign-determining TSTS quadratic form depends only on

\begin{equation}
T_k
=H+k\,g\otimes g
-\frac12(g\otimes\mathbf1+\mathbf1\otimes g).
\end{equation}
Its quadratic form admits the completion

\begin{equation}
\ann{\langle z,T_k.z\rangle}
=\ann{\left\langle z,\left(H-\frac1{4k}\mathbf1\otimes\mathbf1\right).z\right\rangle}
+k\Bigl(g\cdot z-\frac{\mathbf1\cdot z}{2k}\Bigr)^2.
\end{equation}
The last term is always nonnegative. This completion isolates the possible loss of \TSTSM in the volumetric coefficient $\psi''(s)-1/(4k)$. \ann{Substituting the Hessian decomposition gives}

\begin{equation}
H-\frac1{4k}\mathbf1\otimes\mathbf1
=\Delta+\Bigl(\psi''(s)-\frac1{4k}\Bigr)\mathbf1\otimes\mathbf1.
\end{equation}
The elementary inequality

\begin{equation}
\psi''(s)=3\exp(-s)+\exp(s)\ge 2\sqrt3,
\end{equation}
holds for every $s\in\mathbb R$, with equality only at $s=\tfrac12\log 3$. Therefore, if $k\ge 1/(8\sqrt3)$, the coefficient of $\mathbf1\otimes\mathbf1$ is nonnegative. The remaining term $\ann{\langle z,\Delta.z\rangle}=4\sum_i\exp(2x_i)z_i^2$ is strictly positive for $z\neq0$. Hence, for every $k\ge k_*$ and every finite $x\in\mathbb R^3$, the incremental monotonicity pairing is strictly positive. At $k=k_*$ this positivity is pointwise rather than uniform; Section 4 exhibits a sequence for which the reduced quadratic form tends to zero only at infinite distortion. The noncoaxial moduli of the full tensor derivative,

\begin{equation}
\mu\exp(-s+k(p-7))\frac{e^{2x_i}-e^{2x_j}}{x_i-x_j}>0,
\end{equation}
remain positive for every $k>0$, with repeated-stretch limits $2\mu\exp(-s+k(p-7))e^{2x_i}$. Together with these positive noncoaxial moduli, positive definiteness of the principal quadratic form $T_k$ establishes \TSTSM on the full symmetric-matrix space, including repeated eigenvalues.

As an independent check of the principal-tangent calculation, the same threshold is recovered from the invariant sufficient condition of \citet{Wollner2026}. Because $\Psi_k$ is independent of the cofactor invariant $K_2$, Corollary 4.4.1 of that reference reduces the condition to positive-definiteness of a $2\times2$ matrix. Write $C_k$ for the positive exponential prefactor of $\Psi_k$:

\begin{equation}
C_k
=\frac{\mu}{2}\exp\!\left(k\left(K_1^2+\frac3{K_3}+K_3-7\right)\right)>0,
\end{equation}
one has $\partial_{K_1}\Psi_k=2C_k K_1>0$, and direct differentiation gives

\begin{equation}
M_k
=C_k
\begin{pmatrix}
4K_1^2(1+k K_1^2)
&K_1^2\left[2k\left(K_3-\dfrac3{K_3}\right)-1\right]
\\
K_1^2\left[2k\left(K_3-\dfrac3{K_3}\right)-1\right]
&\dfrac6{K_3}+k\left(K_3-\dfrac3{K_3}\right)^2
\end{pmatrix},
\end{equation}
and

\begin{equation}
\frac{\det M_k}{C_k^2}
=K_1^2\left[
\frac{24}{K_3}
+4k\left(K_3-\frac3{K_3}\right)^2
+K_1^2\left(4k\left(K_3+\frac3{K_3}\right)-1\right)
\right].
\end{equation}
Since $K_3+3K_3^{-1}\ge2\sqrt3$, the leading principal minor and the determinant are positive at every state if $k\ge1/(8\sqrt3)$. Conversely, if $0<k<1/(8\sqrt3)$, fix the volume ratio $K_3=\sqrt3$ and increase the distortional deformation along $F=\operatorname{diag}(e^{\ell_{\mathrm d}},e^{-\ell_{\mathrm d}},\sqrt3)$. Then $K_1\to\infty$, and the coefficient of $K_1^4$ is negative. The invariant condition therefore holds globally precisely for $k\ge k_*$. Section 4 supplies the matching finite-deformation loss of \TSTSM, so this bound is the exact constitutive threshold for the present family.

\subsection{Properness and global inversion in the \TSTSM regime}\label{sec:proper}

Let $\mathcal S_k(X):=\widehat\sigma_k(X)=\sigma_k(e^X)$ denote the full tensor stress map on $\operatorname{Sym}(3)$. In the regime $k\ge k_*$, \TSTSM first gives injectivity: for $X\neq Y$,

\begin{equation}
\langle\mathcal S_k(X)-\mathcal S_k(Y),X-Y\rangle
=\int_0^1
\langle \Diff_X\mathcal S_k(Y+t(X-Y)).(X-Y),X-Y\rangle\,\mathrm dt
>0.
\end{equation}
Thus two distinct logarithmic strains cannot produce the same Cauchy stress, including at repeated eigenvalues.

Surjectivity requires one further fact: an unbounded strain cannot approach a finite stress. The Euclidean norm of the principal stress vector equals the Frobenius norm of the full stress tensor. Convexity and $g(0)=0$ give $g(x)\cdot x\ge p(x)-7$, and hence

\begin{equation}
\lVert s_k(x)\rVert
\ge
\frac{\mu}{2}\exp(-s)\exp(k(p(x)-7))\,\frac{p(x)-7}{\lVert x\rVert}.
\end{equation}
If $\lVert x\rVert_\infty\to\infty$, then either some $\exp(2x_i)$ diverges or $s\to-\infty$ and $3\exp(-s)$ diverges, so $p(x)\to\infty$. Moreover,

\begin{equation}
\exp(-s)\ge p^{-1},
\qquad
-2\log p+O(1)\le x_i\le\tfrac12\log p,
\end{equation}
where the lower bound follows from $x_i=s-x_j-x_\ell$. Hence $\lVert x\rVert=O(\log p)$ and

\begin{equation}
\lVert s_k(x)\rVert
\ge
\frac{\mu}{2}\exp(k(p-7))\frac{p-7}{p\lVert x\rVert}
\longrightarrow\infty.
\end{equation}
The full stress map is therefore proper for every $k>0$. In the \TSTSM regime its derivative is invertible, so its image is open. A continuous proper map between finite-dimensional Euclidean spaces is closed, so the image is also closed. Since $\operatorname{Sym}(3)$ is connected, the image is all of $\operatorname{Sym}(3)$. Injectivity and surjectivity complete the proof of Theorem 1.

\par\smallskip
\ann{Polyconvexity and TSTS can also be obtained with the second invariant $I_2=\lVert\operatorname{Cof} F\rVert^2$. Replacing $I_1$ by $I_2$ in the preceding construction requires adjusting the volumetric coefficients to preserve the stress-free reference state. The resulting energy is}
\ann{\begin{equation}\label{eq:cofactor-analogue}
W_k^{I_2}(F)
=\frac{\mu}{2k}
\bigl[\exp\bigl(k(I_2+5J^{-1}+J-9)\bigr)-1\bigr].
\end{equation}}%
\ann{The change from $3J^{-1}$ to $5J^{-1}$ enforces the same stress-free reference state. At $\lambda_1=\lambda_2=\lambda_3=1$, the derivatives of $I_1$ and $I_2$ with respect to each principal stretch are $2$ and $4$, respectively. The corresponding derivatives of the two exponential arguments therefore vanish through $2-3+1=0$ and $4-5+1=0$. The constant $9$ gives zero reference energy. \ann{Likewise, $I_2+5J^{-1}+J-9$ is an objective, isotropic scalar strain measure. It is nonnegative and vanishes exactly for rigid rotations.} Linearization yields the reference shear modulus $\mu_0=\mu$, bulk modulus $\kappa_0=17\mu/3$, and Poisson ratio $\nu_0=(3\kappa_0-2\mu_0)/[2(3\kappa_0+\mu_0)]=5/12$.}

\ann{The representative $(Y,G,\ann{J})\mapsto\lVert G\rVert^2+5\ann{J}^{-1}+\ann{J}-9$ is convex for independent $G$ and $\ann{J}>0$ and independent of $Y$; evaluation at $(F,\operatorname{Cof}F,\det F)$ gives the physical core. Composition with the convex increasing exponential proves polyconvexity for every $k>0$. \ann{For every $k>0$, the Kirchhoff stress associated with $W_k^{I_2}$ also satisfies global strict Hill monotonicity with respect to $\log V$.} For TSTS, write $X=\log V$ and $s=\operatorname{tr}X$. The invariant term becomes $\operatorname{tr}\exp(2(s\Id-X))$. The trace exponential has a positive-definite Hessian on $\operatorname{Sym}(3)$ \citep{LewisSendov2001}. The linear map $X\mapsto s\Id-X$ is invertible, so this invariant term retains that positivity. The volumetric second derivative is $5e^{-s}+e^s\ge2\sqrt5$. The completed-square estimate in Section \ann{3.4} therefore gives the sufficient condition $k\ge1/(8\sqrt5)$ for global strict \TSTSM{}, including equality. \ann{The energy $W_k^{I_2}$ satisfies global \TSTSM{} if and only if $k\ge1/(8\sqrt5)$. Since it is polyconvex for every $k>0$, this is also its exact hyper-Hooke parameter range.} Section~\ref{sec:scope} discusses the shear, torsional and acoustic responses that can test this choice of invariant.}

\ann{Combining $W_k$ and $W_k^{I_2}$ with positive coefficients allows the reference Poisson ratio to vary. Let $h\in C^\infty(0,\infty)$ be convex with $h(1)=h'(1)=0$, and set}
\begin{equation*}
\ann{W_{\mathrm{mix}}(F)=\alpha W_k(F)+\beta W_k^{I_2}(F)+h(J),
\qquad \alpha,\beta>0,\quad k\ge k_*.}
\end{equation*}
\ann{\ann{The convex representatives and the positive incremental stress pairings add with positive weights. Together with convexity of $h$, this makes $W_{\mathrm{mix}}$ a hyper-Hooke model for $\alpha,\beta>0$ and $k\ge k_*$.} With $\eta=h''(1)\ge0$, its reference shear and bulk moduli are $\mu_0=(\alpha+\beta)\mu$ and $\kappa_0=(8\alpha+17\beta)\mu/3+\eta$. The reference Poisson ratio is
$\nu_0=((2\alpha+5\beta)\mu+\eta)/(2((3\alpha+6\beta)\mu+\eta))$.
Thus $1/3<\nu_0<1/2$ for strictly positive $\alpha,\beta$ and finite $\eta$. If $\beta=0$ is allowed, then $\nu_0=1/3$ when $\eta=0$. The value $\nu_0=1/2$ is approached in the incompressible constraint limit.}

\section{Proof of Theorem 2: sharp threshold for \TSTSM}

The preceding argument proves \TSTSM for $k\ge k_*$. To show that the bound cannot be lowered, it is enough to find a finite-strain path and an incremental direction along which the monotonicity pairing becomes negative. Consider the fixed-volume path

\begin{equation}
x_{\ell_{\mathrm d}}=\bigl(\ell_{\mathrm d},-\ell_{\mathrm d},\tfrac12\log 3\bigr),
\qquad \ell_{\mathrm d}\in\mathbb R.
\end{equation}
The three principal stretches along this path are $(e^{\ell_{\mathrm d}},e^{-\ell_{\mathrm d}},\sqrt3)$, and the volume ratio remains $J=\sqrt3$. As $|\ell_{\mathrm d}|$ increases, one direction extends while another contracts, so the distortional deformation grows at fixed volume ratio. Along the path one has $s=\tfrac12\log 3$, so the Arithmetic-Geometric Mean (AM-GM) bound on $\psi''$ is saturated and $\exp(s)-3\exp(-s)=0$. Therefore $\psi''(s)=2\sqrt3$ and $g=2(\exp(2\ell_{\mathrm d}),\exp(-2\ell_{\mathrm d}),3)$. If $0<k<k_*$, the defect

\begin{equation}
\delta:=\frac1{4k}-2\sqrt3,
\end{equation}
is strictly positive. The following test increment is chosen so that the nonnegative square in the completed form decays rapidly, exposing the negative volumetric contribution:

\begin{equation}
z_{\ell_{\mathrm d}}=\Bigl(\frac{\exp(-2\ell_{\mathrm d})}{4k},\,1,\,0\Bigr)\ann{.}
\end{equation}
\ann{This test increment permits a volume change, since $\ann{\langle\mathbf1,z_{\ell_{\mathrm d}}\rangle}=1+e^{-2\ell_{\mathrm d}}/(4k)>0$, even though the base-state path has fixed volume.}
\ann{Substitution in the completed-square identity gives}

\begin{equation}
\ann{\langle z_{\ell_{\mathrm d}},T_k.z_{\ell_{\mathrm d}}\rangle}
=\ann{\langle z_{\ell_{\mathrm d}},\Delta.z_{\ell_{\mathrm d}}\rangle}
-\delta\,(\ann{\langle\mathbf1,z_{\ell_{\mathrm d}}\rangle})^2
+k\Bigl(\ann{\langle g,z_{\ell_{\mathrm d}}\rangle}-\frac{\ann{\langle\mathbf1,z_{\ell_{\mathrm d}}\rangle}}{2k}\Bigr)^2.
\end{equation}
As $\ell_{\mathrm d}\to\infty$ the first term is $O(\exp(-2\ell_{\mathrm d}))$ and the last term is $O(\exp(-4\ell_{\mathrm d}))$, while $(\ann{\langle\mathbf1,z_{\ell_{\mathrm d}}\rangle})^2\to1$. The reduced quadratic form therefore tends to $-\delta<0$. The prefactor $A_k(x_{\ell_{\mathrm d}})$ is strictly positive, so the reduced form and the full monotonicity pairing have the same sign. For every $k<k_*$, all sufficiently large but finite distortions on this path give a negative pairing. Thus $k_*$ is both sufficient and necessary for global \TSTSM, proving Theorem 2; Figure 1\ann{(d)} displays the sign separation across the threshold.

\section{Proof of Theorem 3: sharp threshold for stress bijectivity}

Although \TSTSM and stress bijectivity involve the same principal tangent, they test different constitutive properties. \TSTSM requires its symmetric part $T_k$ to be positive definite. A nonsingular unsymmetrized tangent gives local $C^1$ invertibility, whereas global bijectivity additionally requires one-to-one attainment of every stress and can survive an isolated tangent singularity. Loss of \TSTSM below $k_*$ therefore does not by itself imply multiple stretches; the bijectivity threshold must be determined separately.

\subsection{Statewise tangent-singularity indicator}

Write the unsymmetrized factor of the principal tangent as

\begin{equation}
L_k=H+g\otimes(k g-\mathbf1),
\end{equation}
\ann{Here $\mathbf1=(1,1,1)^{\mathsf T}$ as in Section 2.} \ann{Thus} $\Diff_x s_k=A_k L_k$. If $g=0$, then $L_k=H\succ0$; this occurs only at the reference stretch. For $g\neq0$ set

\begin{equation}
a_{\mathrm{inv}}=\ann{\langle g,H^{-1}.g\rangle},
\qquad
b_{\mathrm{inv}}=\ann{\langle g,H^{-1}.\mathbf1\rangle},
\qquad
\theta(x)=\frac{b_{\mathrm{inv}}-1}{a_{\mathrm{inv}}}.
\end{equation}
The rank-one determinant formula (matrix determinant lemma) yields

\begin{equation}
\frac{\det L_k}{\det H}=a_{\mathrm{inv}}\bigl(k-\theta(x)\bigr).
\end{equation}
The scalar $\theta(x)$ is the value of $k$ at which the principal tangent becomes singular at the state $x$. Its largest value over all strain states is therefore the threshold above which the principal tangent is everywhere regular. Global bijectivity is established separately in Section 5.3.

To determine which deformation first develops a fold, set $y_i=\exp(2x_i)$, $J=\exp(s)$, $Y_{\mathrm{inv}}=\sum_i y_i$ and $S_{\mathrm{inv}}=\sum_i y_i^{-1}$, so that $\prod_i y_i=J^2$. With $\psi'(s)=J-3J^{-1}$ and $\psi''(s)=J+3J^{-1}$ one has $g=2y+\psi'(s)\mathbf1$ and $H=\Delta+\psi''(s)\mathbf1\otimes\mathbf1$. Explicit inversion of this rank-one update by the Sherman-Morrison formula gives

\begin{equation}
\theta(x)=\frac{1-3S_{\mathrm{inv}}/J}{2N_{\mathrm{inv}}(Y_{\mathrm{inv}},S_{\mathrm{inv}};J)},
\qquad
N_{\mathrm{inv}}=Y_{\mathrm{inv}}+3\psi'(s)+\frac{\psi'(s)^2S_{\mathrm{inv}}}{4}+\frac{\psi''(s)(S_{\mathrm{inv}}Y_{\mathrm{inv}}-9)}{4},
\qquad
a_{\mathrm{inv}}=\frac{4N_{\mathrm{inv}}}{4+\psi''(s)S_{\mathrm{inv}}}.
\end{equation}
Since $H\succ0$ one has $a_{\mathrm{inv}}>0$ whenever $g\neq0$; because $4+\psi''(s)S_{\mathrm{inv}}>0$, the last identity gives $N_{\mathrm{inv}}>0$ off the reference stretch. A nonpositive numerator therefore yields $\theta\le0$. Suppose now that the numerator is positive. Direct differentiation gives

\begin{equation}
\partial_{Y_{\mathrm{inv}}}N_{\mathrm{inv}}=1+\frac{\psi''(s)S_{\mathrm{inv}}}{4}>0,
\qquad
\partial_{S_{\mathrm{inv}}}N_{\mathrm{inv}}=\frac{\psi'(s)^2+\psi''(s)Y_{\mathrm{inv}}}{4}>0.
\end{equation}
At fixed $J$, the inequalities $Y_{\mathrm{inv}}\ge3J^{2/3}$ and $S_{\mathrm{inv}}\ge3J^{-2/3}$ hold by AM-GM, with simultaneous equality only when $y_1=y_2=y_3$. The hydrostatic state then maximizes the positive numerator and minimizes the positive denominator, hence

\begin{equation}
\theta(x)\le\theta_{\rm h}(s/3),
\qquad
\theta_{\rm h}(t)=\frac{2\ell^3(\ell^5-9)}{3(\ell^6+2\ell^5-3)^2},
\qquad
\ell=e^t.
\end{equation}
This comparison applies in the positive-numerator region. If $s=0$ and $x\neq0$, then $J=1$ and $S_{\mathrm{inv}}\ge3$, so the numerator is strictly negative and $\theta\le0$. For $s\neq0$, combining the two numerator cases gives

\begin{equation}
\theta(x)\le\max\{0,\theta_{\rm h}(s/3)\}
\qquad(s\neq0).
\end{equation}
Because the desired global maximum is positive, every maximizing state must lie in the positive-numerator region and be hydrostatic. The global critical value can therefore be found from the one-dimensional positive maximum of $\theta_{\rm h}$.

At the reference stretch $g=0$, so $\theta$ is not defined and the tangent is already positive definite. For $\ell\neq1$, one has $\theta_{\rm h}\le0$ when $\ell\le9^{1/5}$ and $\theta_{\rm h}\to0$ as $\ell\to\infty$. Its derivative factors through the polynomial $r$ displayed in Theorem 3:

\begin{equation}
\frac{\mathrm d\theta_{\rm h}}{\mathrm d\ell}
=-\frac{2\ell^2r(\ell)}{3(\ell^6+2\ell^5-3)^3},
\qquad
r'(\ell)=2\ell^4(22\ell^6+20\ell^5-243\ell-255)>0
\qquad\text{for }\ell\ge9/5.
\end{equation}
The ordered coefficients of $r$ have one sign change, so Descartes' rule gives one positive root. Since $r(9/5)<0<r(2)$, this root is simple and gives the unique global maximum, with $\theta_{\rm h}''(t_{\rm B})<0$. Consequently,

\begin{equation}
\sup_{x\in\mathbb R^3\setminus\{0\}}\theta(x)=k_{\rm B},
\qquad
\ell_{\rm B}\ann{\approx \ann{1.862}},
\qquad
t_{\rm B}=\log\ell_{\rm B}\ann{\approx \ann{0.622}},
\end{equation}
and the critical state is the hydrostatic stretch $V=\ell_{\rm B}\Id$. The first tangent singularity, equivalently the first loss of local $C^1$ invertibility, therefore occurs under hydrostatic dilation rather than distortional deformation.

\subsection{Hydrostatic response and the three constitutive regimes}

The fold has an immediate mechanical manifestation under hydrostatic deformation. Along $x=(t,t,t)$ one has $g=g_{\mathrm h}(t)\mathbf1$ with $g_{\mathrm h}(t)=2e^{2t}-3e^{-3t}+e^{3t}$. Denote the common value of the three principal Cauchy stresses by $\sigma_{\rm h}(t)$. By continuous extension at $t=0$, its tangent satisfies

\begin{equation}
\frac{\mathrm d\sigma_{\rm h}}{\mathrm dt}
=3A_k g_{\mathrm h}(t)^2\bigl(k-\theta_{\rm h}(t)\bigr).
\end{equation}
The factor $g_{\mathrm h}$ vanishes only at the identity, while $\sigma_{\rm h}\to-\infty$ as $t\to-\infty$ and $\sigma_{\rm h}\to+\infty$ as $t\to+\infty$. For $0<k<k_{\rm B}$, the equation $\theta_{\rm h}(t)=k$ has two roots. The hydrostatic stress consequently has two turning points, and a finite hydrostatic-stress interval contains three distinct volumetric strain states. These turning points are the folds shown in Figure 1\ann{(c)}. Below $k_{\rm B}$, the hydrostatic response is already not injective, and therefore neither is $V\mapsto\sigma(V)$. At $k=k_{\rm B}$ the hydrostatic curve remains strictly increasing, but its first two derivatives vanish at $t=t_{\rm B}$. For $k>k_{\rm B}$ the hydrostatic tangent stays positive. These one-dimensional facts fix the three constitutive regimes of Theorem 3; the full-tensor statements require a separate inversion argument.

\subsection{\texorpdfstring{Global inversion of the Cauchy-stress map}{Global inversion of the Cauchy-stress map}}

We now distinguish set-theoretic bijectivity from smooth invertibility. The Cauchy-stress relation is bijective for $k\ge k_{\rm B}$; a differentiable inverse exists only for $k>k_{\rm B}$.

For $k>k_{\rm B}$, the bound $\theta(x)<k$ makes the principal tangent nonsingular for every $x\in\mathbb R^3$. Together with the noncoaxial moduli of Section \ann{3.4}, the full derivative $\Diff_X\mathcal S_k$ is invertible at every state. The stress response is smooth, and the logarithm identifies $\operatorname{Sym}^{++}(3)$ with $\operatorname{Sym}(3)$, so $\mathcal S_k$ is a local $C^\infty$ diffeomorphism from $\operatorname{Sym}(3)$ to itself. Section \ann{3.5} gives properness. Its image is open and closed and hence equals $\operatorname{Sym}(3)$; a proper local diffeomorphism is then a covering of the target. Because the target is simply connected and the domain is connected, the covering has one sheet. The stress response is therefore a global $C^\infty$ diffeomorphism.

At $k=k_{\rm B}$, the equality $\theta(x)=k_{\rm B}$ holds only at $X_{\rm B}=t_{\rm B}\Id$. Write $A_{\rm B}:=A_{k_{\rm B}}(t_{\rm B}\mathbf1)$. Away from $X_{\rm B}$ the full derivative remains invertible. At the critical state only the hydrostatic mode degenerates; the derivative on the five-dimensional deviatoric space $\operatorname{Dev}(3)$ is $4A_{\rm B}\ell_{\rm B}^2I_{\operatorname{Dev}}$ and is strictly positive.

Decompose the input as $X=t\Id+X_{\mathrm d}$ with $X_{\mathrm d}\in\operatorname{Dev}(3)$, and decompose the output into its hydrostatic coordinate and deviatoric stress $\Sigma_{\mathrm d}$. By continuity, after shrinking to a product neighborhood $I_t\times B_{\mathrm d}$, the deviatoric Jacobian satisfies $\det \Diff_{X_{\mathrm d}} \Sigma_{\mathrm d}>0$ throughout. The implicit-function theorem then makes $(t,\Sigma_{\mathrm d})$ valid local coordinates and reduces the full map to

\begin{equation}
(t,\Sigma_{\mathrm d})\longmapsto(f(t,\Sigma_{\mathrm d}),\Sigma_{\mathrm d}).
\end{equation}
The determinant identity above, together with the positive noncoaxial factors, gives

\begin{equation}
\det \Diff\mathcal S_{k_{\rm B}}(X)\ge0,
\end{equation}
with equality only at $X_{\rm B}$. Choosing the local coordinates with consistent orientation, the Schur-complement formula yields

\begin{equation}
\partial_tf(t,\Sigma_{\mathrm d})
=c\,\frac{\det \Diff\mathcal S_{k_{\rm B}}}{\det \Diff_{X_{\mathrm d}} \Sigma_{\mathrm d}}\ge0,
\qquad c>0,
\end{equation}
and the inequality is strict except at $(t,\Sigma_{\mathrm d})=(t_{\rm B},0)$. For $t_1<t_2$ in $I_t$, integration of $\partial_tf$ over $[t_1,t_2]$ is therefore strictly positive, i.e.\ the integrand can vanish at no more than one point. Hence $f(\cdot,\Sigma_{\mathrm d})$ is strictly increasing for every fixed $\Sigma_{\mathrm d}$, so the full map is locally injective. Invariance of domain now applies and gives a local homeomorphism at the critical state.

Along the hydrostatic line,

\begin{equation}
\sigma_{\rm h}'''(t_{\rm B})
=-3A_{\rm B}g_{\mathrm h}(t_{\rm B})^2\theta_{\rm h}''(t_{\rm B})>0.
\end{equation}
Thus the inverse of the hydrostatic restriction has cube-root behavior. Since this one-dimensional inverse is the restriction of the full inverse to the hydrostatic stress line, the full inverse cannot be differentiable at the critical stress. The stress response is a local homeomorphism at every state, so its image is open; properness makes the image closed. Connectedness of $\operatorname{Sym}(3)$ therefore gives surjectivity. A proper local homeomorphism between locally compact Hausdorff spaces is a covering; here both spaces are finite-dimensional Euclidean spaces. Connectedness of the domain together with simple connectedness of the target then give one sheet. The map is therefore a global homeomorphism at $k=k_{\rm B}$. This completes the proof of Theorem 3.

\section{\ann{Constitutive scope and modelling limitations}}\label{sec:scope}

\ann{This section discusses the scope and modelling limitations of the proposed energy. The preceding results first determine the elastic response at a material point: global stress inversion assigns a unique positive-definite stretch to every prescribed symmetric Cauchy stress. In the \TSTSM{} regime the inverse is smooth, so small changes in prescribed stress give smoothly varying recovered stretches and a well-defined inverse constitutive tangent. This provides a globally unique elastic response for stress-controlled constitutive calculations. Assessing its physical scope also requires examining local stability under nonuniform deformation and the stress and structural responses it produces in specific loading problems.}

\ann{Local stability requires considering nonuniform perturbations near the material point. \ann{Section 3.1 proves strict Legendre-Hadamard ellipticity for every $k>0$ and every $F\in\mathrm{GL}^+(3)$: \begin{equation}\label{eq:strict-lh}
\left\langle\Diff_F^2W_k(F).(\newann{\xi}\otimes\newann{\eta}),
\newann{\xi}\otimes\newann{\eta}\right\rangle>0,
\qquad \newann{\xi},\newann{\eta}\in\mathbb R^3\setminus\{0\}.
\end{equation}} Thus $W_k$ remains strongly elliptic and cannot describe localization initiated by loss of local ellipticity (see \citealp{KnowlesSternberg1978} for finite plane elasticity). Describing that onset requires an incremental stiffness that reaches zero, corresponding to degeneration of the acoustic tensor. A subsequent branch with a negative Legendre-Hadamard form requires relaxing global polyconvexity. Within the present family, lowering $k$ or adding convex volumetric terms preserves strict ellipticity. This limitation follows from the rank-one stiffness of the energy, whereas TSTS concerns the Cauchy-stress tangent in logarithmic-strain space. The exponentiated Hencky energy can satisfy TSTS while losing rank-one convexity \citep{Neff2015ExpHencky}, so these two incremental conditions must be examined separately.}

\ann{\ann{For a hyper-Hooke model, the pressure-radius relation also depends on deformation compatibility, geometry and boundary conditions.} For $W_k$, Theorem B.1 combines precisely these ingredients to establish pressure-radius inversion for the radial annulus. Other energies can produce a nonmonotone pressure branch while satisfying global \TSTSM{}, as the exponentiated Hencky annulus in Appendix B illustrates. \ann{The compressible $I_1$ energy in Appendix C.6 is itself a hyper-Hooke model with strong ellipticity and stress bijectivity, yet it admits a descending pressure interval under uniform incompressible thin-spherical-membrane inflation.} Along this interval the hoop-radial stress difference continues to increase, but its growth does not compensate for the effect of membrane thinning and radius growth on pressure. These results show why both the growth of material stress and the deformed geometry enter the assessment of pressure turning points.}

\ann{The $I_1$ example in Appendix C.6 shows that a descending spherical pressure branch can be obtained within the generalized neo-Hookean class. That class also imposes restrictions on shear, torsion and elastic-wave responses, which must be examined when choosing an invariant dependence for a particular material. Under incompressibility, $W_k$ also belongs to this class. In simple shear $F=\Id+\gamma e_1\otimes e_2$ with $\sigma_{33}=0$, the shearing-face normal stress is $\sigma_{22}=-2\gamma^2\partial W/\partial I_2$. An $I_1$-only energy therefore gives $\sigma_{22}=0$, and adding $I_2$ dependence changes this component directly. In torsion of an incompressible circular cylinder held at unit axial stretch, with a traction-free lateral surface, the same class gives $N=-\omega_{\mathrm t}M/2$. Here $\omega_{\mathrm t}$ is the twist per unit length (with units of inverse length) and $N,M$ are the axial force and torque, using the sign convention of \citet{AnssariBenamBucchiSaccomandi2021I2}. Thus specifying the torque at a given twist also fixes the axial force. The measured departures discussed in that study restrict the ability of an $I_1$-only energy to fit these responses simultaneously.}

\ann{Elastic-wave measurements further test the restriction that the same invariant dependence places on higher-order elastic constants. For two-dimensional shear waves, the generalized neo-Hookean class imposes $\beta_2=1+\mathcal A/(4\mu_0)=0$, or $\mathcal A=-4\mu_0$, where $\mathcal A$ is the third-order Landau constant. The reference shear modulus therefore also fixes this third-order constant. The measurements compiled by \citet{DestradePucciSaccomandi2019} show departures from that relation. Together with the shear and torsion responses above, these observables can test whether an $I_2$-dependent extension, such as the energy in~\eqref{eq:cofactor-analogue}, has sufficient freedom to describe a particular material and can constrain its calibration.}

\ann{The preceding tests constrain the macroscopic stress response. Relating $W_k$ to a molecular model further requires identifying the deformation measured by $I_1$ and the origin of its hardening function. For a reference unit direction $n$, let $\langle\cdot\rangle_{S^2}$ denote the uniform directional average. Since $\langle n\otimes n\rangle_{S^2}=\Id/3$,}
\ann{$$
\left\langle |Fn|^2\right\rangle_{S^2}=\frac{I_1}{3}.
$$}%
\ann{Thus $\sqrt{I_1/3}$ measures the root-mean-square line stretch, also for compressible deformations. This measure gives a direct connection to existing molecular-chain descriptions.}

\ann{In the affine-network description discussed by \citet{AnssariBenam2021Molecular}, the effective chain stretch is $\lambda_c=\sqrt{I_1/3}$. Under incompressibility the present energy becomes}
\ann{$$
W_k\big|_{J=1}
=\frac{\mu}{2k}\left[\exp\!\left(3k(\lambda_c^2-1)\right)-1\right].
$$}%
\ann{With principal-stretch exponent two, the models of \citet{AnssariBenamSaccomandi2025Compressible} depend on the same $\lambda_c$ at $J=1$. This common chain-stretch measure allows their stress responses to be compared through the growth of the respective hardening functions. In the limiting-chain regime, the cited energy diverges at a finite value of $\lambda_c$ to describe the extensibility limit of the chains. The energy $W_k$ remains finite at every finite $\lambda_c$. Its parameter $k$ controls exponential hardening without introducing a finite chain-extensibility boundary. Using $W_k$ to describe chain hardening therefore requires testing whether this growth law represents the material response over the stretch range of interest.}

\ann{The common chain-stretch measure permits a comparison of the macroscopic hardening laws, but it does not provide a molecular derivation of $W_k$. Such a derivation requires a single-chain free energy and a network model, whose predictions must then be assessed against the loading tests discussed above.}

\section{Conclusions}

This paper answers affirmatively the question announced at the 2024 IUTAM symposium \citep{Neff2024} and formalized by Wollner, Holzapfel and Neff \citep{Wollner2026}. For every $\mu>0$ and $k>0$, $W_k$ is defined on all of $\mathrm{GL}^+(3)$, polyconvex and strictly rank-one convex, with a stress-free reference state and positive infinitesimal shear and bulk moduli. \ann{For $k\ge k_*$, $W_k$ is a hyper-Hooke model whose Cauchy stress maps positive-definite stretches bijectively onto all symmetric stresses.} The construction also meets the invariant conditions of \citet{Wollner2026} without an explicit cofactor term. \ann{\ann{This hyper-Hooke model defines an idealized elastic law over the full deformation domain.} Its global stress inverse is useful in numerical constitutive calculations where the strain range is not known in advance, since every prescribed symmetric Cauchy stress determines a unique positive-definite stretch.} \ann{For $k\ge k_*$ and $\alpha,\beta>0$, the combined energy $W_{\mathrm{mix}}$ remains polyconvex and globally strictly \TSTSM{}. Its reference Poisson ratio satisfies $1/3<\nu_0<1/2$. In the incompressible constraint limit, $\nu_0$ approaches $1/2$.}

\ann{It is interesting to note that the corresponding Kirchhoff stress is globally strictly monotone in logarithmic strain for every $k>0$, and every symmetric Kirchhoff stress determines a unique positive-definite stretch with smooth dependence on stress. These properties coexist with failure of Cauchy-stress \TSTSM{} for $0<k<k_*$.}

As an analytic constitutive benchmark, the family also separates the sharp thresholds for stress bijectivity and \TSTSM. The lower threshold $k_{\rm B}$ governs global inversion: below it the hydrostatic response folds; at it the stress map is a global homeomorphism with a nondifferentiable inverse; above it the map is a global $C^\infty$ diffeomorphism. Since $k_{\rm B}<k_*$, the intervening interval exhibits global stress bijectivity without \TSTSM, showing that bijectivity alone does not ensure the incremental stability encoded by \TSTSM. Appendix A records the closed stress and incremental responses of four homogeneous deformation modes and gives the stress formulas under the isochoric constraint $J=1$. Appendix~B proves that, for $k\ge k_*$, every prescribed inner radius has a unique smooth radial annular equilibrium and the pressure required to maintain the cavity increases strictly from zero to infinity. In the four-energy comparison, CG-NH has a finite pressure ceiling, $W_{\mathrm{eH}}$ has a computed nonmonotone annular branch, and $W_{\exp}$ has a globally strictly convex reduced density and an exponential large-cavity energy bound. For $W_k$, Theorem~B.1 combines global shooting, a unique smooth branch, strict pressure monotonicity, and unbounded growth into global pressure inversion. Under incompressible spherical inflation, global pressure inversion holds for the Demiray family with $b\ge b_*$, including $W_k|_{J=1}$ for $k\ge k_*$ and $W_{\exp}|_{J=1}$, and for the exponentiated Hencky membrane when $k_{\mathrm{eH}}>3/8$; incompressible neo-Hooke instead develops a limit point at $\lambda=7^{1/6}$.
\ann{The spherical limitation is a property of the specified family and parameters: small-parameter Demiray laws can have two pressure turning points. Appendix C.6 gives a compressible $I_1$ energy with the same reference moduli as $W_k$ that retains global polyconvexity, strict ellipticity, \TSTSM and stress bijectivity while exhibiting two spherical pressure turning points. Pressure decreases when the positive stress difference grows more slowly than the cube of the circumferential stretch. These results supply an ideal elastic benchmark. Modelling loss of ellipticity, material failure and the observed shear, torsional or acoustic response requires separate assessment.}

\appendix
\setcounter{figure}{0}
\renewcommand{\thefigure}{A\arabic{figure}}
\renewcommand{\theHfigure}{A\arabic{figure}}
\section{Stress response in four homogeneous deformation modes}

Theorems 1-3 concern the full Cauchy-stress map on $\mathrm{GL}^+(3)$. The four subsections below record the closed Cauchy stress in four homogeneous deformation modes. Hydrostatic Cauchy stress is already shown in Figure 1\ann{(c)}. Throughout the figures we take $\sigma_{\mathrm{ref}}=\mu/2=1$ and the three values $k=k_*$, $2k_*$ and $4k_*$.

The starting point is the principal formula of Section 2,

\begin{equation}
\sigma_i=A_k g_i,
\qquad
A_k=\frac{\mu}{2}\exp\bigl(-s+k(p-7)\bigr)>0,
\qquad
g_i=2\lambda_i^2+J-3J^{-1},
\end{equation}
or, equivalently, the tensor representation

\begin{equation}
\sigma_k(F)
=\frac{\mu}{2}\exp\bigl(k(P-7)\bigr)
\bigl[2J^{-1}B+(1-3J^{-2})\Id\bigr].
\end{equation}
By Theorem 1, $W_k$ is polyconvex and strictly rank-one convex on the whole of $\mathrm{GL}^+(3)$ for every $k>0$. For a $C^1$ principal-logarithmic-strain path $x=\xi(\alpha)$ with fixed principal axes, the full principal-space second-order work is \ann{$A_k\,\langle\xi',T_k.\xi'\rangle$}.
If $k\ge k_*$ and $\xi'\neq0$, \TSTSM makes this quantity strictly positive. Sections A.1, A.2 and A.4 use this pairing. The principal axes rotate in simple shear, so Section A.3 uses the one-dimensional shear-stress tangent. Figure A1 gives the Cauchy-stress responses in the four modes, and Figure A2 gives the incremental quantities that control their strict monotonicity with respect to the loading parameter.

\begin{figure}[ht]
\centering
\includegraphics[width=\textwidth]{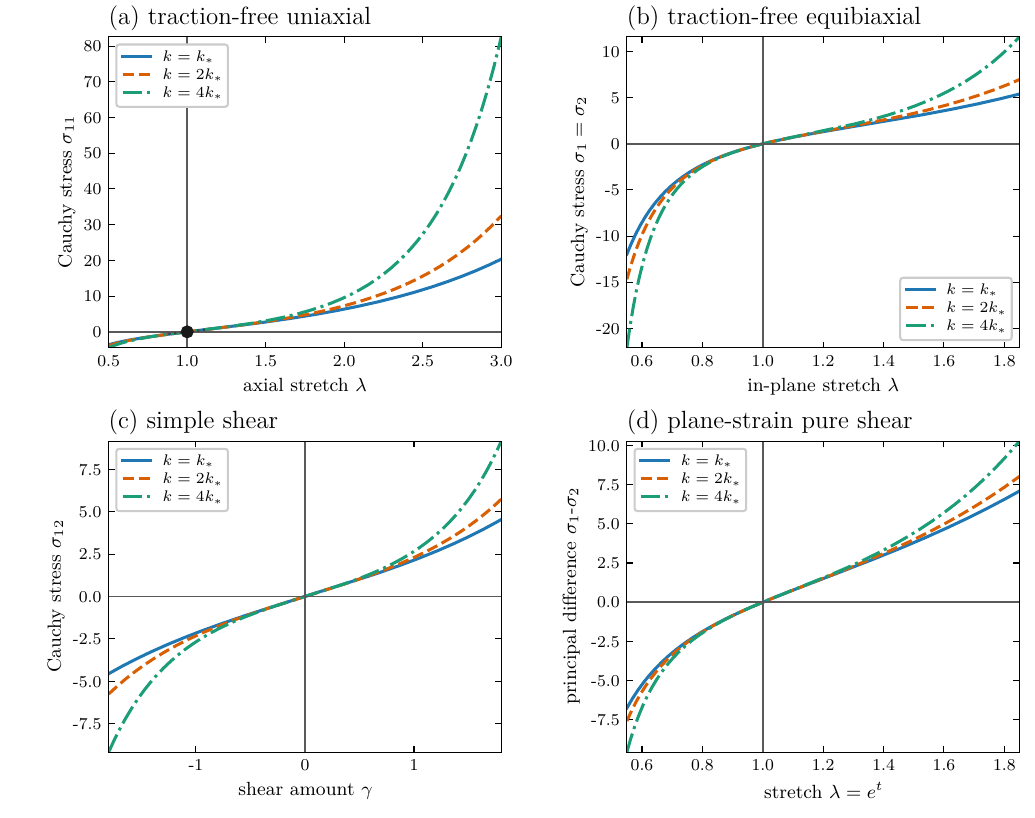}
\caption{Cauchy-stress responses of $W_k$ in four homogeneous deformation modes, with $\sigma_{\mathrm{ref}}=\mu/2=1$. Solid blue, dashed orange, and dash-dotted green correspond to $k=k_*$, $2k_*$, and $4k_*$. Panel (a) shows $\sigma_{11}$ in uniaxial deformation with traction-free lateral faces; the black marker is the stress-free identity and the short dotted segment is the reference-state linearization $\frac{16}{3}(\lambda-1)$. Panel (b) shows $\sigma_1=\sigma_2$ in equibiaxial deformation with a traction-free thickness direction. Panel (c) shows $\sigma_{12}$ in simple shear. Panel (d) shows the principal-stress difference $\sigma_1-\sigma_2$ in plane-strain pure shear, with $\lambda=e^t$ on the abscissa. All four responses pass through the stress-free identity and increase strictly over their full parameter domains, as established by the closed relations in Sections A.1-A.4. The hydrostatic response appears in Figure 1\ann{(c)}, and the simple-shear Poynting stress $\sigma_{11}$ is given in Section A.3.}
\label{fig:A1}
\end{figure}

\begin{figure}[ht]
\centering
\includegraphics[width=\textwidth]{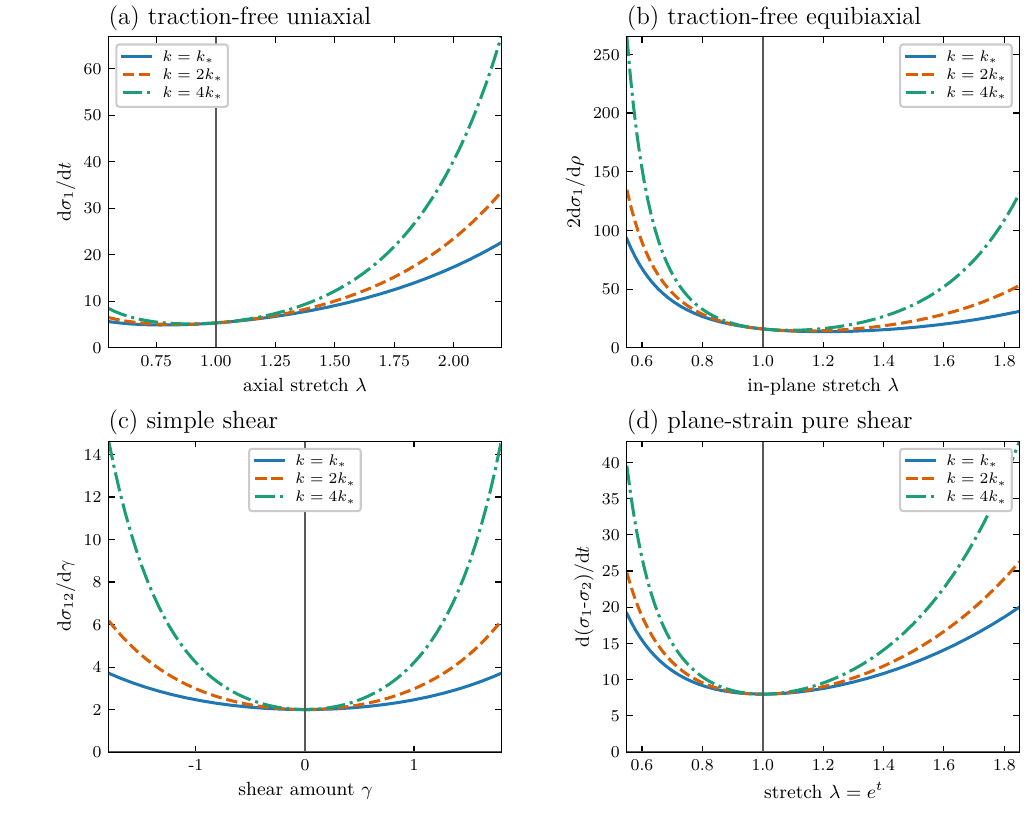}
\caption{Incremental responses along the four loading paths of Figure A1, with $\sigma_{\mathrm{ref}}=\mu/2=1$. Solid blue, dashed orange, and dash-dotted green correspond to $k=k_*$, $2k_*$, and $4k_*$. The ordinates are the increments identified in Sections A.1-A.4. Panel (a) shows $\mathrm d\sigma_1/\mathrm dt=A_k\ann{\langle\xi',T_k.\xi'\rangle}$. Panel (b) shows $2\mathrm d\sigma_1/\mathrm d\rho$. Panel (c) shows $\mathrm d\sigma_{12}/\mathrm d\gamma=\mu\exp(k\gamma^2)(1+2k\gamma^2)$. Panel (d) shows $\mathrm d(\sigma_1-\sigma_2)/\mathrm dt$. Positivity in panels (a) and (b) follows from \TSTSM for $k\ge k_*$; the closed expressions in panels (c) and (d) are strictly positive for every $k>0$. These four positive incremental Cauchy moduli \citep{Neff2025CSP} directly yield the strict monotonicity shown in Figure A1.}
\label{fig:A2}
\end{figure}

\subsection{Compressible uniaxial deformation}

Consider traction-free uniaxial deformation

\begin{equation}
F=\operatorname{diag}(\lambda,r,r),\qquad \sigma_2=\sigma_3=0,\qquad \lambda>0.
\end{equation}
The volume ratio is $J=\lambda r^2$. Because $A_k>0$, the lateral condition $\sigma_2=\sigma_3=0$ is equivalent to $g_2=g_3=0$, that is

\begin{equation}
2r^2+J-3J^{-1}=0.
\end{equation}
Set $y=r^2>0$, so that $J=\lambda y$. Substitution gives $2y+\lambda y-3/(\lambda y)=0$, hence $y^2(\lambda+2)=3/\lambda$. The positive root is

\begin{equation}
y(\lambda)=\sqrt{\frac{3}{\lambda(\lambda+2)}},\qquad
J(\lambda)=\sqrt{\frac{3\lambda}{\lambda+2}}.
\end{equation}
The same identity rewrites $3J^{-1}=2y+J$, so that along this deformation $P$ reduces to

\begin{equation}
P=\lambda^2+2y+3J^{-1}+J=\lambda^2+4y+2J.
\end{equation}
The axial factor is $g_1=2\lambda^2+J-3J^{-1}=2(\lambda^2-y)$, and the axial Cauchy stress has the closed form

\begin{equation}
\sigma_{11}^{\mathrm{comp}}(\lambda)
=\frac{\mu}{J(\lambda)}\exp\bigl(k(P-7)\bigr)\bigl(\lambda^2-y(\lambda)\bigr).
\end{equation}
At $\lambda=1$ one has $y=1$, $J=1$, $P=7$ and $\sigma_{11}^{\mathrm{comp}}=0$. Differentiating the closed form, or invoking the linearized moduli of Section \ann{3.3}, gives the reference slope

\begin{equation}
\frac{\mathrm d\sigma_{11}^{\mathrm{comp}}}{\mathrm d\lambda}\Big|_{\lambda=1}=\frac{8\mu}{3},
\end{equation}
which is the Young modulus $E_Y=8\mu/3$ of the infinitesimal law; the constitutive parameter is the shear modulus $\mu$, and $\sigma_{\mathrm{ref}}=\mu/2$ is the plotting stress scale.

Parametrize the path by $t=\log\lambda$ and write $\chi_{\mathrm u}(t)$ for the lateral Hencky strain. Then

\begin{equation}
\chi_{\mathrm u}(t)=\frac14\log\frac{3}{e^{t}(e^{t}+2)},
\qquad
\chi_{\mathrm u}'(t)=-\frac12\frac{e^{t}+1}{e^{t}+2}.
\end{equation}
\Needspace{6\baselineskip}
Along $\xi(t)=(t,\chi_{\mathrm u}(t),\chi_{\mathrm u}(t))$ the second-order work of uniaxial loading equals the axial stress slope in the logarithmic parametrization,

\begin{equation}
A_k\,\ann{\langle\xi',T_k.\xi'\rangle}=\frac{\mathrm d\sigma_1}{\mathrm dt}.
\end{equation}
Figures~\ref{fig:A1}(a) and \ref{fig:A2}(a) show the axial Cauchy stress and its logarithmic-strain tangent, respectively. Restricting the same energy $W_k$ to the isochoric constraint $J=1$ gives $I_1=\lVert F\rVert^2$ and

\begin{equation}
\left.W_k(F)\right|_{J=1}
=\frac{\mu}{2k}\left[\exp\bigl(k(I_1-3)\bigr)-1\right].
\end{equation}
With $b=k$, this expression coincides with the Demiray/Fung exponential energy \citep{Fung1967,Demiray1972}

\begin{equation}
W_{\mathrm D}(F)
=\frac{\mu}{2b}\left[\exp\bigl(b(I_1-3)\bigr)-1\right].
\end{equation}
For isochoric uniaxial deformation $F_{\mathrm{inc}}=\operatorname{diag}(\lambda,\lambda^{-1/2},\lambda^{-1/2})$, elimination of the constraint pressure by the lateral traction-free condition gives

\begin{equation}
\sigma_{11}^{\mathrm{inc}}(\lambda)
=\mu\exp\bigl(k(\lambda^2+2\lambda^{-1}-3)\bigr)
\bigl(\lambda^2-\lambda^{-1}\bigr),
\end{equation}
with reference slope $3\mu$. This restriction to $J=1$ gives the analytic connection between $W_k$ and the Demiray/Fung type.

\subsection{Traction-free equibiaxial deformation}

\ann{Now take $F=\operatorname{diag}(\lambda,\lambda,u)$ with $\sigma_3=0$ and $\lambda>0$, so that $J=\lambda^2u$.} The thickness condition $g_3=0$ reads

\begin{equation}
2u^2+J-3J^{-1}=0,
\end{equation}
or equivalently

\begin{equation}
u^2\bigl(2\lambda^2 u+\lambda^4\bigr)=3.
\end{equation}
For $u>0$ the left-hand side is strictly increasing from $0$ to $\infty$, so a unique positive thickness exists for every $\lambda>0$. In logarithmic coordinates $\rho=\log\lambda$ and $\chi_{\mathrm b}=\log u$, implicit differentiation of the residual yields

\begin{equation}
\chi_{\mathrm b}'(\rho)=-\frac{2\bigl(e^{2\rho}u+e^{4\rho}\bigr)}{3e^{2\rho}u+e^{4\rho}},
\end{equation}
and in particular $\chi_{\mathrm b}'(0)=-1$. The in-plane stresses coincide. Using $g_3=0$ to replace $J-3J^{-1}$ by $-2u^2$ gives $g_1=g_2=2(\lambda^2-u^2)$ and

\begin{equation}
\sigma_1=\sigma_2
=\frac{\mu}{J}\exp\bigl(k(P-7)\bigr)\bigl(\lambda^2-u^2\bigr),
\qquad
P=2\lambda^2+3u^2+2J.
\end{equation}
Along $\xi(\rho)=(\rho,\rho,\chi_{\mathrm b}(\rho))$ one has $\xi'=(1,1,\chi_{\mathrm b}')$. Because $\sigma_3\equiv0$ in this deformation, the full principal-space second-order work in equibiaxial loading is

\begin{equation}
A_k\,\ann{\langle\xi',T_k.\xi'\rangle}=2\frac{\mathrm d\sigma_1}{\mathrm d\rho}.
\end{equation}
Figures~\ref{fig:A1}(b) and \ref{fig:A2}(b) show the equibiaxial in-plane stress and the corresponding incremental response, respectively.

\subsection{Simple shear}

Let $\gamma\in\mathbb R$ denote the shear amount and set
\begin{equation}
F=\begin{pmatrix}1&\gamma&0\\0&1&0\\0&0&1\end{pmatrix}.
\end{equation}
Then $J=1$ and $\lVert F\rVert^2=3+\gamma^2$, so that $P=7+\gamma^2$. The left Cauchy-Green tensor is

\begin{equation}
B=FF^{\mathsf T}
=\begin{pmatrix}
1+\gamma^2&\gamma&0\\
\gamma&1&0\\
0&0&1
\end{pmatrix}.
\end{equation}
The tensor formula therefore collapses to $\sigma=\mu\exp(k\gamma^2)\,(B-\Id)$. The only nonzero Cauchy components are the shear and Poynting stresses

\begin{equation}
\sigma_{12}=\mu\gamma\exp(k\gamma^2),
\qquad
\sigma_{11}=\mu\gamma^2\exp(k\gamma^2).
\end{equation}
In particular, $\sigma_{12}$ is odd in $\gamma$ while $\sigma_{11}$ is even and nonnegative. Differentiating the shear stress gives

\begin{equation}
\frac{\mathrm d\sigma_{12}}{\mathrm d\gamma}
=\mu\exp(k\gamma^2)\bigl(1+2k\gamma^2\bigr)>0,
\end{equation}
for every $k>0$ and every real $\gamma$. Figures~\ref{fig:A1}(c) and \ref{fig:A2}(c) show the shear stress and its tangent, respectively. This scalar shear-stress curve is therefore strictly increasing.

\subsection{Plane-strain pure shear}

Let the principal stretches be $(e^{t},e^{-t},1)$. \ann{Then $J=1$ and $P-7=e^{2t}+e^{-2t}-2=4\sinh^2t$.}
The principal Cauchy stresses satisfy $g_1-g_2=2e^{2t}-2e^{-2t}=4\sinh(2t)$, hence

\begin{equation}
\sigma_1-\sigma_2
=2\mu\exp\bigl(4k\sinh^2 t\bigr)\sinh(2t).
\end{equation}
Differentiating with respect to $t$ produces the closed slope

\begin{equation}
\frac{\mathrm d}{\mathrm dt}(\sigma_1-\sigma_2)
=2\mu\exp\bigl(4k\sinh^2 t\bigr)
\bigl(4k\sinh^2(2t)+2\cosh(2t)\bigr).
\end{equation}
Figures~\ref{fig:A1}(d) and \ref{fig:A2}(d) show the principal-stress difference and its derivative, respectively. Along $\xi(t)=(t,-t,0)$ one has $\xi'=(1,-1,0)$, and the full principal-space second-order work satisfies

\begin{equation}
A_k\,\ann{\langle\xi',T_k.\xi'\rangle}
=\frac{\mathrm d}{\mathrm dt}(\sigma_1-\sigma_2)>0,
\end{equation}
for every $k>0$ and $t\in\mathbb R$.

Appendix A treats homogeneous deformation. Appendix~B proves Theorem~B.1 for the pressure-controlled plane-strain annulus and compares $W_k$ with CG-NH, $W_{\mathrm{eH}}$, and $W_{\exp}$ at radial-stress inversion, reduced convexity, global shooting, pressure monotonicity, and endpoint growth. Appendix~C proves global pressure-stretch inversion for the incompressible Demiray family with $b\ge b_*$ and derives the limit point of incompressible neo-Hooke.

\setcounter{figure}{0}
\renewcommand{\thefigure}{B\arabic{figure}}
\renewcommand{\theHfigure}{B\arabic{figure}}
\setcounter{theorem}{0}
\renewcommand{\thetheorem}{B.\arabic{theorem}}
\renewcommand{\theHtheorem}{B.\arabic{theorem}}
\section{Unique radial equilibrium and pressure-radius inversion in a plane-strain annulus}

In finite elasticity, a plane-strain annulus with a traction-free outer wall and a controlled inner radius or inner pressure is a classical semi-inverse boundary-value problem \citep{Ciarlet1988}. The reference body is the ring $A\le R\le B$, the axial stretch is fixed at $1$, and the radial deformation is described by $r(R)$. Prescribing the current inner radius $r(A)=a$ and imposing $\sigma_r(B)=0$ make the cavity pressure $p_i=-\sigma_r(A)$ the reaction required to maintain that deformation. Because the hoop stretch $u=r/R$ and radial stretch $v=r'$ generally differ, the solution requires one to recover $v$ uniquely from $(u,\sigma_r)$ at each material point and then integrate radial equilibrium.

Among the four energies, only $W_k$ is both polyconvex and globally \TSTSM. This appendix asks what complete result can be established for an annular boundary-value problem governed by this energy. For $W_k$ with $k\ge k_*$, Theorem~B.1 gives a unique smooth radial equilibrium and proves that the inner pressure increases strictly from zero to infinity with the current cavity radius. Polyconvexity supplies convexity in the minors, the global constitutive property used in standard variational existence theory for finite elasticity, while \TSTSM directly gives $\partial\sigma_r/\partial\log v>0$ in the radial problem. The complete result additionally uses strict convexity of the plane-strain reduction of $W_k$, a dedicated positive-definiteness argument for the pressure-derivative quadratic form, global shooting estimates, and exponential growth.

Three comparator energies show the distinct roles of these properties in the annular problem. The compressible Ciarlet--Geymonat neo-Hooke (CG-NH) law is polyconvex, but its plane-strain reduced Hessian changes sign. The exponentiated Hencky energy $W_{\mathrm{eH}}$ with $k_{\mathrm{eH}}=0.40>3/8$ satisfies \TSTSM but is not globally rank-one convex. The stress-free $I_1$ exponential $W_{\exp}$ is polyconvex and has a strictly convex plane-strain reduction, but fails \TSTSM at finite strain. Section~B.2 gives the reduced formulas and Hessians for all four energies.

\subsection{Boundary-value problem and main theorem}

We first specify the annular kinematics, boundary conditions, and observable load, and then state the conclusions proved in the subsequent subsections. Let $0<A<B<\infty$ and impose the radial plane-strain ansatz

\begin{equation}
r=r(R),
\qquad
F=\operatorname{diag}\left(r'(R),\frac{r(R)}R,1\right).
\end{equation}
Write

\begin{equation}
v=r'(R)>0,
\qquad
u=\frac{r(R)}R>0,
\qquad
J=vu.
\end{equation}
Figure~\ref{fig:B1}(a) shows this reference annulus, with the axial stretch fixed at $1$. \ann{The inner radius is prescribed as $r(A)=a\ge A$, the outer wall is traction free, $\sigma_r(B)=0$, and the required inner pressure is $p_i(a)=-\sigma_r(A)$.}
Figure~\ref{fig:B1}(b) shows the deformed annulus, the cavity pressure $p_i$, and the traction-free outer wall $\sigma_r(B)=0$. The equilibrium analysis below uses precisely these kinematics and boundary conditions.

\begin{figure}[ht]
\centering
\includegraphics[width=0.82\textwidth]{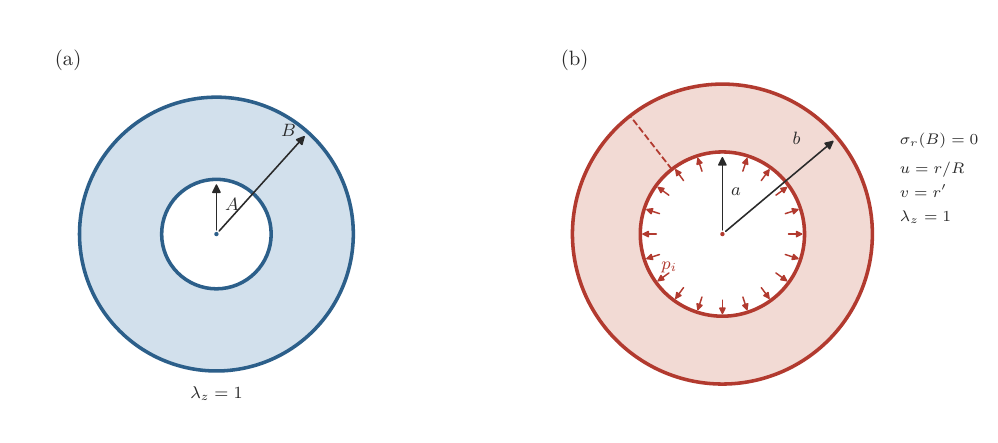}
\caption{Kinematics and boundary data for the radial plane-strain annulus. Panel (a) shows the reference body $A\le R\le B$ with axial stretch $\lambda_z=1$. Panel (b) shows the deformed radii $a=r(A)$ and $b=r(B)$, the cavity pressure $p_i=-\sigma_r(A)$, the traction-free outer wall $\sigma_r(B)=0$, and the principal stretches $(v,u,1)=(r',r/R,1)$.}
\label{fig:B1}
\end{figure}

\begin{theorem}\label{thm:B1} Let $0<A<B<\infty$, $\mu>0$, and $k\ge k_*$. For every prescribed inner radius $a\ge A$, the radial plane-strain boundary-value problem for $W_k$ with $r_a(A)=a$ and $\sigma_r(B)=0$ has exactly one $C^\infty$ solution $r_a$, and $r_a'(R)>0$ for every $R\in[A,B]$. Among radial plane-strain deformations with the same inner-boundary displacement, $r_a$ is the unique minimizer of the total strain energy, and it depends smoothly on $a$. The corresponding inner pressure satisfies

\begin{equation}
p_i(A)=0,
\qquad
\frac{dp_i}{da}>0,
\qquad
p_i(a)\to\infty
\quad(a\to\infty).
\end{equation}
\ann{Consequently, $p_i:[A,\infty)\to[0,\infty)$}
is a smooth, strictly increasing bijection. Its inverse $a=a(p_i)$ is smooth for $p_i>0$ and admits a one-sided smooth extension at $p_i=0$.

\end{theorem}

Theorem B.1 follows from four linked mechanisms. Section~B.2 combines the radial-tangent positivity supplied by \TSTSM with the endpoint stress limits to obtain pointwise radial-stress inversion, and separately proves strict convexity of the plane-strain reduction. Section~B.3 uses global orbit estimates specific to $W_k$ to establish existence. Section~B.4 then obtains variational uniqueness and smooth dependence from reduced strict convexity and invertibility of the Jacobi operator. Section~B.5 combines the TSTS tangent term with the geometric term in the annular pressure derivative and proves by explicit algebra for $W_k$ that the resulting matrix $M_{\mathrm{ann}}$ is positive definite. Section~B.6 uses exponential growth to exclude a finite pressure ceiling and thereby completes global pressure-radius inversion.

\subsection{Plane-strain energy and radial-stress inversion}

Global \TSTSM first gives radial-tangent positivity at fixed $u$ on the plane-strain slice. Pointwise inversion additionally requires the radial stress to cover the real line, whereas the subsequent variational uniqueness requires joint strict convexity of the two-variable reduced energy. We therefore restrict the three-dimensional energy to $F=\operatorname{diag}(v,u,1)$ and establish these two properties separately. For CG-NH, $W_{\mathrm{eH}}$, and $W_{\exp}$, we likewise calculate the radial inversion and reduced Hessian to distinguish pointwise inversion from joint convexity. The plane-strain energy density is

\begin{equation}
w_k(v,u)
=\frac{\mu}{2k}
\left[
\exp\!\left(k\Phi(v,u)\right)-1
\right],
\qquad
\Phi(v,u)
=v^2+u^2+1+\frac3{vu}+vu-7.
\end{equation}
We use the lower-semicontinuous extension $w_k=+\infty$ whenever $v\le0$ or $u\le0$. The Hessian of its kernel $\Phi$ satisfies

\begin{equation}
\ann{\Phi_{,vv}=2+\frac{6}{v^3u}>0,\qquad \det\Diff_{(v,u)}^2\Phi=\frac{3(J^2+3)^2}{J^4}+\frac{12(v-u)^2}{J^3}>3.}
\end{equation}
Hence

\begin{equation}
\Diff_{(v,u)}^2 w_k
=\frac{\mu}{2}e^{k\Phi}
\left(\Diff_{(v,u)}^2\Phi+k\Diff_{(v,u)}\Phi\otimes\Diff_{(v,u)}\Phi\right)\succ0
\end{equation}
throughout the positive quadrant. The rank-one update cannot destroy this positivity.

Under the standard calibration \citep{CiarletGeymonat1982}, the three-dimensional Ciarlet--Geymonat neo-Hooke energy is
\begin{equation}
W_{\mathrm{CG\text{-}NH}}(F)
=\mu\left[
\frac12\lVert F\rVert^2
+\frac7{12}J^2
-\frac{13}{6}\log J
-\frac{25}{12}
\right],
\qquad
J=\det F.
\end{equation}
Restricted to the plane-strain stretch $F=\operatorname{diag}(v,u,1)$, this becomes

\begin{equation}
w_{\mathrm{CG\text{-}NH}}(v,u)
=\mu\left[
\frac12(v^2+u^2+1)
+\frac7{12}(vu)^2
-\frac{13}6\log(vu)
-\frac{25}{12}
\right],
\end{equation}
whose Hessian determinant with respect to $(v,u)$ is

\begin{equation}
\frac1{\mu^2}\det\Diff_{(v,u)}^2 w_{\mathrm{CG\text{-}NH}}(v,u)
=\left(1+\frac76u^2+\frac{13}{6v^2}\right)
\left(1+\frac76v^2+\frac{13}{6u^2}\right)
-\frac{49}{9}v^2u^2.
\end{equation}
The remaining comparison energies are the exponentiated Hencky energy \citep{Neff2015ExpHencky} and a stress-free $I_1$ exponential. Their three-dimensional forms, with $k_{\mathrm{eH}}=0.40$, are
\begin{equation}
W_{\mathrm{eH}}(F)
=\frac{\mu}{k_{\mathrm{eH}}}
\exp\bigl(\,k_{\mathrm{eH}}\lVert\log V\rVert^2\bigr),
\qquad
W_{\exp}(F)
=\frac{\mu}{2}
\bigl[\exp(\lVert F\rVert^2-2\log\det F-3)-1\bigr].
\end{equation}
Restricted to the plane-strain stretch $F=\operatorname{diag}(v,u,1)$, they become

\begin{equation}
w_{\mathrm{eH}}(v,u)
=\frac{\mu}{k_{\mathrm{eH}}}
\exp\!\left(k_{\mathrm{eH}}\bigl[(\log v)^2+(\log u)^2\bigr]\right),
\end{equation}
\begin{equation}
w_{\exp}(v,u)
=\frac{\mu}{2}\left[\exp\!\left(\phi_{\exp}(v,u)\right)-1\right],
\qquad
\phi_{\exp}=v^2+u^2-2\log(vu)-2.
\end{equation}
Because $\Diff^2\phi_{\exp}=\operatorname{diag}(2+2/v^2,2+2/u^2)\succ0$, the increasing convex outer exponential makes $w_{\exp}$ globally strictly convex on the positive quadrant. This reduced convexity does not restore three-dimensional \TSTSM. The CG-NH determinant above changes sign. For $w_{\mathrm{eH}}$, direct differentiation at the equal-stretch state $(v,u)=(e^2,e^2)$ gives
\begin{equation}
\frac1{\mu^2}\det\Diff_{(v,u)}^2w_{\mathrm{eH}}(e^2,e^2)
=-\frac{108}{5}e^{-8/5}<0,
\end{equation}
so its plane-strain restriction is not globally convex even though the three-dimensional Cauchy stress satisfies \TSTSM. Figure~\ref{fig:B2} plots the same normalized Hessian determinant for all four energies.

\begin{figure}[ht]
\centering
\includegraphics[width=\textwidth]{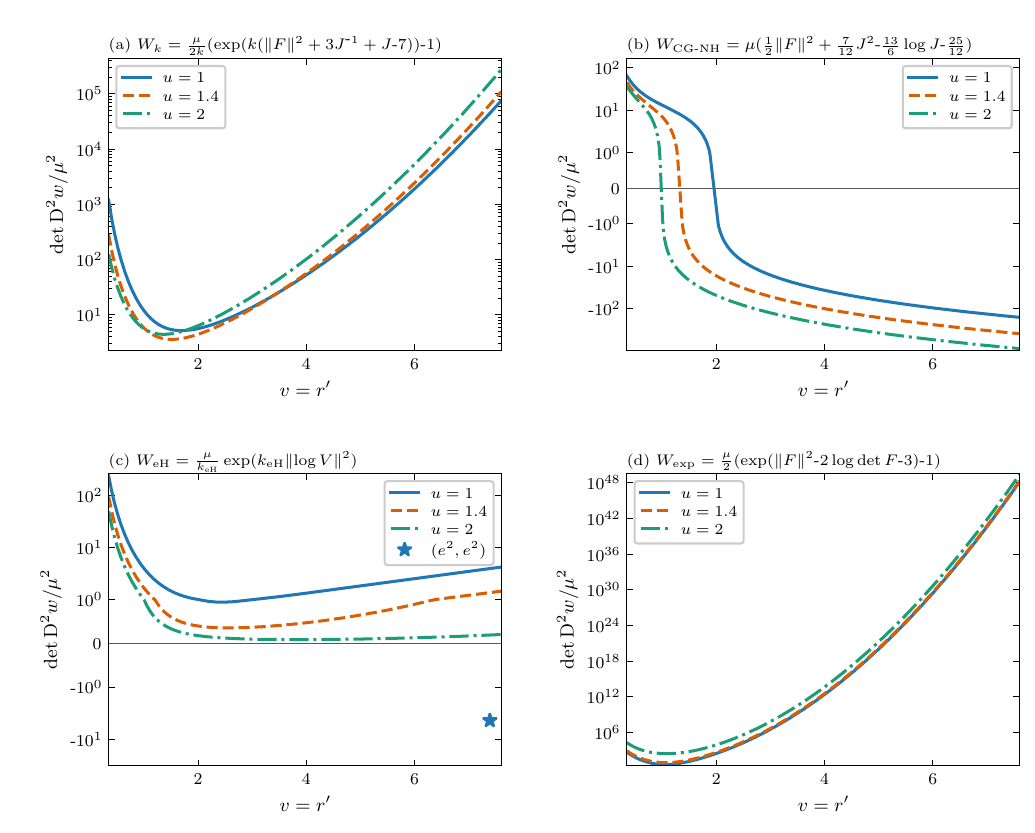}
\caption{Reduced plane-strain Hessian comparison. Every panel plots the same dimensionless quantity $\det\mathrm{D}^2w/\mu^2$ against $v=r'$ on the common slices $u=1,1.4,2$ and the common window $0.35\le v\le7.6$. The positive panels (a) and (d) use logarithmic ordinates; the sign-changing panels (b) and (c) use symmetric-logarithmic ordinates, with the horizontal line marking zero. Panel titles give the four energies; (a) uses $k=k_*$ and (c) uses $k_{\mathrm{eH}}=0.40$. The star in panel (c) is the analytic equal-stretch witness $(v,u)=(e^2,e^2)$, where the determinant equals $-(108/5)e^{-8/5}$. The reduced Hessians of $W_k$ and $W_{\exp}$ are positive definite globally, whereas the determinants for CG-NH and $W_{\mathrm{eH}}$ change sign.}
\label{fig:B2}
\end{figure}

For fixed $u>0$, the radial Cauchy stress satisfies

\begin{equation}
\frac{d\sigma_r}{d\log v}=\ann{(\mathcal T_k)_{11}}>0,
\end{equation}
\ann{Here $\mathcal T_k=A_kT_k=\operatorname{sym}(\Diff_xs_k)$ is the dimensional symmetric principal stress tangent, and $T_k$ is the reduced matrix of Section \ann{3.4}.} Moreover,

\begin{equation}
\sigma_r(v,u)\to-\infty
\quad(v\downarrow0),
\qquad
\sigma_r(v,u)\to+\infty
\quad(v\to\infty).
\end{equation}
Thus $v\mapsto\sigma_r(v,u)$ is a $C^\infty$ diffeomorphism $(0,\infty)\to\mathbb R$. \ann{Denote its inverse by $v=\widehat v(u,s_r)$, where $s_r=\sigma_r$.}
\ann{The free-face relation $s_r=0$ reduces to $2uv^3+u^2v^2=3$.}
Its left-hand side is strictly increasing from zero to infinity in $v>0$, so the positive free-face stretch is unique for every $u>0$.

Thus, for every $(u,s_r)$, the radial stretch $v=\widehat v(u,s_r)$ is uniquely determined, and the free outer wall $s_r=0$ likewise selects a unique radial stretch. This supplies the local stress inversion required by radial equilibrium.

All three comparators also admit pointwise radial-stress inversion. For $W_{\mathrm{eH}}$, it follows from global \TSTSM. For CG-NH and $W_{\exp}$, direct calculation gives, respectively,
\begin{equation}
\ann{\frac1\mu\frac{\partial\sigma_r^{\mathrm{CG\text{-}NH}}}{\partial\log v}=\frac vu+\frac76vu+\frac{13}{6vu}>0,\qquad
\frac{J}{\mu e^{\phi_{\exp}}}\frac{\partial\sigma_r^{\exp}}{\partial\log v}=2v^4-3v^2+3>0.}
\end{equation}
\ann{The second inequality follows from $2v^4-3v^2+3=2(v^2-3/4)^2+15/8>0$.}
In each case the radial stress tends to $-\infty$ as $v\downarrow0$ and to $+\infty$ as $v\to\infty$. Thus all four energies admit pointwise radial-stress inversion, while global joint convexity holds for $W_k$ and $W_{\exp}$. For $W_k$, global annular solvability and pressure monotonicity require the dedicated arguments of Sections~B.3-B.5. Section B.3 inserts $\widehat v$ for $W_k$ into the equilibrium equations and integrates inward from the traction-free outer wall.

\subsection{Inward integration from the traction-free outer wall: global existence}

The pointwise inversion of Section~B.2 converts radial equilibrium into a smooth first-order system, but local inversion alone does not ensure that an orbit reaches the inner wall or covers every prescribed cavity radius. For $W_k$, we use the outer-wall hoop stretch $u_B=u(B)$ as a parameter, integrate inward from $s_r(B)=0$, prove that the trajectory cannot blow up before reaching $R=A$, and show that the resulting family covers every $a\ge A$. Set $t=\log(R/A)$ and use overdots in this subsection for derivatives with respect to $t$. Radial equilibrium becomes

\begin{equation}
\dot u=\widehat v(u,s_r)-u,
\end{equation}
\begin{equation}
\dot s_r
=2A_k\frac{\widehat v(u,s_r)}u
\left[u^2-\widehat v(u,s_r)^2\right],
\qquad
A_k
=\frac{\mu}{2u\widehat v(u,s_r)}e^{k\Phi}.
\end{equation}
Start at the free outer wall with

\begin{equation}
u(B)=u_B\ge1,
\qquad
s_r(B)=0,
\end{equation}
and integrate towards $A$. The line $v=u$ is an equilibrium line. Uniqueness of the autonomous system prevents a non-equilibrium orbit from crossing it. In reverse time

\begin{equation}
t_R=\log\frac BR,
\qquad
y=-s_r,
\end{equation}
the choice $u_B=1$ gives the identity orbit $u=v=1$. For $u_B>1$, one obtains

\begin{equation}
0<u_{t_R}=u-v<u,
\qquad
u_B\le u(t_R)\le u_Be^{t_R}\le u_B\frac BA.
\end{equation}
Along a non-equilibrium reverse orbit, $y>0$. The constitutive formulas give

\begin{equation}
y=\frac{A_k}{vu}
\left(3-2uv^3-u^2v^2\right),
\qquad
y_{t_R}=2A_k\frac vu\left(u^2-v^2\right).
\end{equation}
Consequently,

\begin{equation}
\frac{y_{t_R}}{y}
=\frac{2v^2(u^2-v^2)}
{3-2uv^3-u^2v^2}
\longrightarrow0
\qquad(v\downarrow0).
\end{equation}
The bound on $u$ makes $y_{t_R}/(1+y)$ bounded over the remaining part of the orbit. \ann{Hence, for a constant $C_{u_B}>0$ depending on $u_B$, $A$, and $B$, it holds that $y_{t_R}\le C_{u_B}(1+y)$.}
Gronwall's inequality excludes finite-time blow-up. Every $u_B\ge1$ therefore generates a positive orbit reaching $R=A$.

Let $\Psi(u_B)=u(A)$.
Continuous dependence and the preceding inequalities give

\begin{equation}
\Psi(1)=1,
\qquad
\Psi(u_B)\ge u_B.
\end{equation}
The connected image $\Psi([1,\infty))$ contains the value $1$ and is unbounded, while remaining in $[1,\infty)$. Hence it equals $[1,\infty)$. For every $a/A\ge1$, at least one strictly orientation-preserving radial field exists.

The inward solution therefore reaches $R=A$ for every outer-wall parameter $u_B\ge1$, and every prescribed inner radius $a\ge A$ has at least one radial equilibrium. Existence does not yet exclude two outer-wall states from reaching the same inner radius.

Pointwise inversion gives analogous first-order shooting systems for the three comparators. For $W_k$, the bounded ratio $y_{t_R}/(1+y)$ and the surjective outer-wall parametrization establish global shooting coverage. Figure~\ref{fig:B3} reports numerically integrated comparator branches; Sections~B.5 and B.6 compare their computed pressure slopes and the analytic endpoint results available for CG-NH and $W_{\exp}$. Section B.4 uses strict convexity to prove uniqueness and the linearized boundary-value problem to prove smooth dependence on $a$.

\begin{figure}[ht]
\centering
\includegraphics[width=\textwidth]{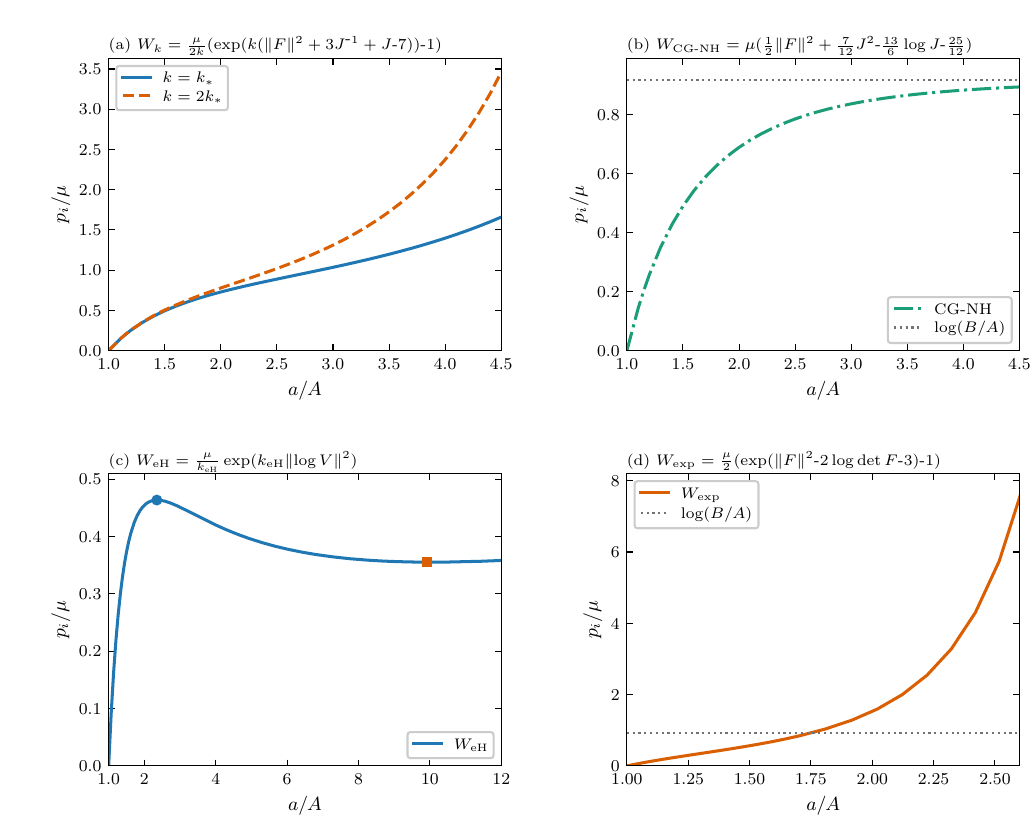}
\caption{Four-energy pressure comparison for the same annulus, $A=1$ and $B=2.5$. Panel titles give the four energies. Every ordinate is $p_i/\mu$, where $\mu$ is the infinitesimal shear modulus of the corresponding energy. Panel (a) shows computed $W_k$ curves for $k=k_*$ and $k=2k_*$ on $1\le a/A\le4.5$; Theorem~B.1 proves that both are globally smooth, strictly increasing, and unbounded. Panel (b) shows the computed CG-NH branch on the same window and the exact upper bound $\log(B/A)$. Panel (c) shows dense reverse-integration data for $W_{\mathrm{eH}}$ at $k_{\mathrm{eH}}=0.40$ on $1\le a/A\le12$; the circle and square mark the computed local maximum $(2.350,\ann{0.464})$ and local minimum $(9.912,\ann{0.355})$. Panel (d) shows the computed $W_{\exp}$ branch on $1\le a/A\le2.60$ together with the CG-NH ceiling; the curve continues beyond the displayed window to the value reported in the text. The CG-NH bound and the endpoint conclusions for $W_k$ and $W_{\exp}$ are analytic. The $W_{\mathrm{eH}}$ extrema and all displayed branch samples are numerical.}
\label{fig:B3}
\end{figure}
\FloatBarrier

\subsection{Energy minimization, uniqueness, and smooth dependence}

Section~B.3 constructed at least one radial equilibrium with the required boundary data. Radial uniqueness uses the strict convexity of the plane-strain reduction proved separately in Section~B.2. Strict convexity gives the unique global energy minimizer, while invertibility of the Jacobi operator gives smooth dependence on the prescribed inner radius. Define

\begin{equation}
\ann{\mathcal A_a=\{r\in W^{1,2}(A,B)\mid r(A)=a,\ r'\ge0\ \text{a.e.}\},\qquad
\mathcal I[r]=2\pi\int_A^B R\,w_k(r',r/R)\,dR.}
\end{equation}
The extended-value convention makes $\mathcal I[r]=+\infty$ if $(r',r/R)\notin(0,\infty)^2$ on a set of positive measure. The orbit constructed above satisfies the Euler-Lagrange equation and the natural outer condition. Strict convexity gives, for every $\widetilde r\in\mathcal A_a$,

\begin{equation}
\mathcal I[\widetilde r]
\ge
\mathcal I[r_a]
+\left\langle\Diff\mathcal I[r_a],\widetilde r-r_a\right\rangle
=\mathcal I[r_a],
\end{equation}
where the pairing denotes the first variation of the energy functional acting on the admissible variation. Equality holds only for $\widetilde r=r_a$. Thus $r_a$ is the unique global radial minimizer. Since its orbit remains in a compact subset of $v,u>0$, standard ODE regularity gives

\begin{equation}
r_a\in C^\infty([A,B]).
\end{equation}
For variations with $\zeta(A)=0$, the linearized natural boundary condition is

\begin{equation}
\left[
w_{k,vv}\zeta'
+w_{k,vu}\frac{\zeta}{R}
\right]_{R=B}=0.
\end{equation}
The Jacobi operator is

\begin{equation}
\mathcal L_a\zeta
=-\left[
R\left(
w_{k,vv}\zeta'
+w_{k,vu}\frac{\zeta}{R}
\right)
\right]'
+w_{k,uv}\zeta'
+w_{k,uu}\frac{\zeta}{R},
\end{equation}
and its quadratic form is

\begin{equation}
2\pi\int_A^B
R\,\left\langle\Diff_{(v,u)}^2 w_k.
\left(\zeta',\frac\zeta R\right),
\left(\zeta',\frac\zeta R\right)
\right\rangle\mathrm dR.
\end{equation}
It is strictly positive unless $\zeta=0$. The uniformly elliptic mixed-boundary operator therefore has no kernel; its Fredholm index is zero, so it is an isomorphism. \ann{The Banach implicit-function theorem yields a local $C^\infty$ map $a\mapsto r_a$ in every $C^{m,\alpha_{\mathrm H}}$ space with integer $m\ge2$ and $\alpha_{\mathrm H}\in(0,1)$} \citep{Ciarlet1988}. Existence and uniqueness patch the local branches over $[A,\infty)$. At $a=A$, the same invertible operator extends the identity branch to a two-sided neighborhood.

The equilibrium constructed in Section B.3 is therefore the unique global radial energy minimizer, and $a\mapsto r_a$ has the stated smooth dependence in Hölder spaces.

Figure~\ref{fig:B2} determines the applicability of the strict-convexity and Jacobi-form argument. It applies to $W_k$ and $W_{\exp}$: any admissible critical point is variationally unique, and its Jacobi form has no kernel. The sign-changing reduced Hessians of CG-NH and $W_{\mathrm{eH}}$ preclude this argument. For $W_k$, the global shooting coverage of Section~B.3 and Jacobi invertibility give the unique smooth branch on $[A,\infty)$. For $W_{\exp}$, this reduced-convexity result does not establish global shooting coverage or pressure monotonicity. Section B.5 differentiates the $W_k$ solution family with respect to $a$ to determine whether the pressure required to maintain the cavity increases strictly with its radius.

\subsection{Strict increase of inner pressure with current cavity radius}

Section~B.4 produced a unique equilibrium family depending smoothly on $a$. We now differentiate both the solution and total energy with respect to $a$. The resulting pressure derivative contains both the stress-tangent term controlled by \TSTSM and the Baker-Ericksen term generated by the annular geometry. Their explicit combination and the ensuing algebraic estimate for $W_k$ prove that the full quadratic form is positive definite when $k\ge k_*$. \ann{Smooth dependence defines $\eta_a=\partial_a r_a$,}
and, in the current principal frame,

\begin{equation}
z=(z_r,z_\theta,0)
=\left(\frac{\eta_a'}{r_a'},\frac{\eta_a}{r_a},0\right).
\end{equation}
For the value function $\mathcal V(a)=\mathcal I[r_a]$, differentiation under the integral, followed by the Euler--Lagrange equation and the conditions $\sigma_r(B)=0$ and $\eta_a(A)=1$, cancels the interior terms. The remaining inner-boundary contribution is \ann{$-2\pi A\,w_{k,v}(A)\eta_a(A)$}. Since $w_{k,v}=u\sigma_r$ and $u(A)=a/A$, this reduces to

\begin{equation}
\mathcal V'(a)=2\pi a\,p_i(a).
\end{equation}
The current-volume form of the second variation is

\begin{equation}
\mathcal V''(a)
=\int_{\rm wall}
\left[
\ann{\langle z,\mathcal T_k.z\rangle}
+(\sigma_r+\sigma_\theta)z_rz_\theta
\right]dv,
\end{equation}
\ann{Here $\mathcal T_k=A_kT_k$ is the dimensional symmetric true-stress--Hencky-strain tangent defined in Section B.2.} On the other hand,

\begin{equation}
\mathcal V''(a)
=2\pi p_i(a)+2\pi a\frac{dp_i}{da}.
\end{equation}
Let $b=r_a(B)$ and define the spatial sensitivity $\overline\eta_a(r):=\eta_a(r_a^{-1}(r))$. In the following spatial integrals the stresses are also evaluated at $R=r_a^{-1}(r)$. Then $\overline\eta_a(a)=\eta_a(A)=1$, and the outer condition is $\sigma_r(b)=0$ in these coordinates. Radial equilibrium $\mathrm d(r\sigma_r)/\mathrm dr=\sigma_\theta$ gives

\begin{equation}
2\pi p_i
=2\pi\int_a^b
\frac{\sigma_\theta-\sigma_r}{r}\overline\eta_a^2\,dr
+4\pi\int_a^b
\sigma_r\overline\eta_a\,\partial_r\overline\eta_a\,dr.
\end{equation}
Substitution into the second variation yields

\begin{equation}
2\pi a\frac{dp_i}{da}
=
\int_{\rm wall}
\left[
\ann{\langle z,\mathcal T_k.z\rangle}
+(\sigma_\theta-\sigma_r)z_\theta(z_r-z_\theta)
\right]dv.
\end{equation}
The Baker-Ericksen remainder is not sign definite by itself \citep{TruesdellNoll1965}. For $W_k$, its combination with the TSTS term is represented by the explicit symmetric matrix below; proving this matrix positive definite further uses the specific algebraic structure of $W_k$. The volume curvature of Section \ann{3.4} satisfies

\begin{equation}
\psi''(s)=e^{s}+3e^{-s}\ge 2\sqrt{3},\qquad s=\log J,
\end{equation}
and therefore $c:=\psi''(s)-\frac1{4k}\ge 0$ whenever $k\ge k_*$. Put

\begin{equation}
\chi_r=v^2,
\qquad
\chi_\theta=u^2,
\qquad
\ann{\chi}=J-\frac3J-\frac1{2k},
\end{equation}
\begin{equation}
\alpha=2\chi_r+\ann{\chi},
\qquad
\beta=2\chi_\theta+\ann{\chi}.
\end{equation}
After division of the current-volume density by the positive factor $A_k$, the reduced quadratic form is represented by

\begin{equation}
M_{\mathrm{ann}}=
\begin{pmatrix}
4\chi_r+c+k\alpha^2 &
\chi_\theta-\chi_r+c+k\alpha\beta\\
\chi_\theta-\chi_r+c+k\alpha\beta &
2\chi_r+2\chi_\theta+c+k\beta^2
\end{pmatrix}.
\end{equation}
\ann{Since $c\ge0$, we have $(M_{\mathrm{ann}})_{11}\ge4v^2>0$.}
Direct algebra gives

\begin{equation}
\det M_{\mathrm{ann}}
=\frac{8(k Q_{\mathrm{ann}}+6J\chi_r^3)}{J^2\chi_r^2},
\qquad
Q_{\mathrm{ann}}=v^5\mathcal R(v,u),
\end{equation}
where

\begin{equation}
\mathcal R(v,u)
=v^7u^2+2v^6u^3+3v^5u^4+2v^4u^5+v^3u^6+3u^5-3v^4u-6v^3u^2-12v^2u^3+9v.
\end{equation}
Set $t=u/v>0$ and $w=v^4>0$. Then

\begin{equation}
\mathcal R=v(A_{\mathrm{ann}}w^2+B_{\mathrm{ann}}w+9),
\end{equation}
\begin{equation}
A_{\mathrm{ann}}=t^2(1+t+t^2)^2,
\qquad
B_{\mathrm{ann}}=3t\,\iota(t),
\qquad
\iota(t)=t^4-4t^2-2t-1.
\end{equation}
The discriminant factors exactly as

\begin{equation}
B_{\mathrm{ann}}^2-36A_{\mathrm{ann}}
=9t^2(t-1)^2(t+1)^2K_{\mathrm{ann}}(t),
\end{equation}
\begin{equation}
K_{\mathrm{ann}}(t)=t^4-6t^2-4t-3,
\qquad
K_{\mathrm{ann}}(t)-\iota(t)=-2(t^2+t+1).
\end{equation}
If $K_{\mathrm{ann}}(t)<0$ and $t\ne1$, the discriminant is negative, so the quadratic is strictly positive. At $t=1$ it reduces to $9(w-1)^2$. If $K_{\mathrm{ann}}(t)\ge0$, then $\iota(t)=K_{\mathrm{ann}}(t)+2(t^2+t+1)>0$, so $B_{\mathrm{ann}}>0$ and $A_{\mathrm{ann}}w^2+B_{\mathrm{ann}}w+9\ge9$. Thus

\begin{equation}
\mathcal R(v,u)\ge0.
\end{equation}
Because $J>0$ and $\chi_r>0$,

\begin{equation}
k Q_{\mathrm{ann}}+6J\chi_r^3>0,
\qquad
\det M_{\mathrm{ann}}>0.
\end{equation}
Sylvester's criterion yields $M_{\mathrm{ann}}\succ0$ throughout the positive quadrant. Since $\eta_a(A)=1$, the incremental field induced by varying the cavity radius cannot vanish identically. Therefore

\begin{equation}
\frac{dp_i}{da}>0
\qquad(a>A).
\end{equation}
Figure~\ref{fig:B3}(a) shows this strictly increasing inner pressure for $W_k$ at $k=k_*$ and $k=2k_*$.
At the identity, the exact plane-strain Lamé slope is

\begin{equation}
p_i'(A)
=\frac{6\mu\bigl((B/A)^2-1\bigr)}
{A\bigl(1+3(B/A)^2\bigr)}>0.
\end{equation}
We have proved $p_i'(a)>0$ for every $a\ge A$, so inner pressure increases strictly with current cavity radius.

The pressure-slope comparison separates \TSTSM from annular pressure monotonicity. The exponentiated Hencky energy satisfies both \TSTSM and the Baker-Ericksen inequalities, while dense reverse integration gives a local maximum at $(a/A,p_i/\mu)\ann{\approx}(2.350,\ann{0.464})$ and a subsequent local minimum at $(9.912,\ann{0.355})$ on the computed branch in Figure~\ref{fig:B3}(c). Along this branch, the positive-definite combination $M_{\mathrm{ann}}$ derived above is therefore not implied by \TSTSM alone. The computed $W_{\exp}$ branch in Figure~\ref{fig:B3}(d) is increasing on the sampled interval although $W_{\exp}$ fails global \TSTSM.

Strict monotonicity gives a one-to-one pressure-radius relation for $W_k$, but surjectivity onto all nonnegative pressures additionally requires ruling out a finite limiting pressure as $a\to\infty$. Section B.6 uses the growth of the energy to prove $p_i(a)\to\infty$ and then compares the endpoint information available for all four energies.

\subsection{Unbounded pressure growth and global inversion}

Sections~B.3-B.5 established a unique smooth branch and a positive pressure slope everywhere, but global pressure inversion also requires its range not to terminate at a finite value. We now use the exponential growth of $W_k$ to derive a large-cavity energy bound, exclude a finite pressure ceiling, and complete the inversion from prescribed radius to prescribed pressure. Strict convexity of the density and affine dependence of the prescribed boundary value imply convexity of \ann{$\mathcal V(a)=\mathcal I[r_a]$.}
\ann{Moreover, $\mathcal V^{\prime}(a)=2\pi a p_i(a)$ and $\mathcal V(A)=0$.}
\ann{Since $r_a(R)\ge a$, we have $r_a(R)/R\ge a/B$. For $a>B\sqrt6$, the energy lower bound and convexity, $\mathcal V^{\prime}(a)\ge\mathcal V(a)/(a-A)$, give}
\begin{equation}
\ann{\begin{aligned}
\mathcal V(a)&\ge\frac{\pi\sigma_{\mathrm{ref}}(B^2-A^2)}{k}
\left[e^{k((a/B)^2-6)}-1\right],\\
p_i(a)&\ge\frac{\sigma_{\mathrm{ref}}(B^2-A^2)}{2ka(a-A)}
\left[e^{k((a/B)^2-6)}-1\right]\longrightarrow\infty
\quad(a\to\infty).
\end{aligned}}
\end{equation}
Together with $p_i(A)=0$, $p_i'(a)>0$, and smooth dependence, this shows that $p_i$ is a smooth, strictly increasing bijection from $[A,\infty)$ onto $[0,\infty)$ and completes the proof of Theorem B.1. Figure~\ref{fig:B3}(a) shows computed samples for $W_k$ at $k=k_*$ and $k=2k_*$. All four panels use the common shear-modulus normalization $p_i/\mu$ with $A=1$ and $B=2.5$. At $a/A=4.5$, the two $W_k$ values are $\ann{1.660}$ and $\ann{3.460}$; the theorem and the lower bound above show that both curves continue to rise and diverge.

The endpoint comparison gives three distinct outcomes. For CG-NH, radial equilibrium yields the exact identity
\begin{equation}
\frac{p_i}{\mu}
=\int_A^B\frac1R\left(1-\frac{v^2}{u^2}\right)dR
\le\log\frac BA.
\end{equation}
Thus its attainable pressure is bounded above. Figure~\ref{fig:B3}(b) shows the computed branch approaching \ann{$\log(2.5)\approx0.916$}; at $a/A=4.5$, \ann{$p_i/\mu\approx0.893$}. For $W_{\exp}$, strict convexity makes the radial value function convex, and for $a\ge B$ its growth obeys
\begin{equation}
\mathcal V_{\exp}(a)
\ge\pi(B^2-A^2)\,
w_{\exp}\!\left(1,\frac aB\right).
\end{equation}
Together with $\mathcal V_{\exp}'(a)=2\pi a p_i(a)$, this lower bound implies $p_i(a)\to\infty$ along any smooth minimizing branch that extends to arbitrarily large $a$. Figure~\ref{fig:B3}(d) shows the computed branch crossing the CG-NH ceiling and continuing to $p_i/\mu\ann{\approx}1.695\times10^5$ at $a/A\ann{\approx}4.397$. For $W_{\mathrm{eH}}$, Figure~\ref{fig:B3}(c) gives the computed extrema described in Section~B.5, and its far-field asymptotic is undetermined.

Among the four energies, only $W_k$ is both polyconvex and globally \TSTSM, and Theorem~B.1 gives it a complete pressure-radius inversion for the annulus. Polyconvexity supplies the convexity-in-minors condition used in standard variational existence theory, while radial uniqueness uses the strict convexity of the plane-strain reduction proved separately in Section~B.2. Global \TSTSM supplies pointwise radial-tangent positivity, and the explicit algebra for $W_k$ in Section~B.5 further combines the TSTS and Baker-Ericksen terms into $M_{\mathrm{ann}}\succ0$. Global shooting, Jacobi invertibility, and exponential growth then close existence, smooth uniqueness, and the pressure range, respectively. The three comparators show that these links cannot replace one another. The energy $W_{\exp}$ is polyconvex with a strictly convex reduction but fails global \TSTSM, so the positive-definiteness argument used for $W_k$ in Section~B.5 does not apply to it. The CG-NH energy is polyconvex while its reduced Hessian changes sign, and its pressure has an exact finite ceiling. The energy $W_{\mathrm{eH}}$ satisfies \TSTSM while its reduction is not jointly convex; its computed pressure branch is nonmonotone, and its far-field asymptotic is undetermined.

Appendix C turns to a different boundary-value problem, incompressible spherical inflation. There the Demiray family with $b\ge b_*$ has a globally invertible pressure-stretch relation, whereas incompressible neo-Hooke has a limit point. This result applies to $W_k|_{J=1}$ for $k\ge k_*$ and to $W_{\exp}|_{J=1}$; the exponentiated Hencky energy has the same inversion when $k_{\mathrm{eH}}>3/8$ by a separate closed formula.

\FloatBarrier
\setcounter{figure}{0}
\renewcommand{\thefigure}{C\arabic{figure}}
\renewcommand{\theHfigure}{C\arabic{figure}}
\setcounter{theorem}{0}
\renewcommand{\thetheorem}{C.\arabic{theorem}}
\renewcommand{\theHtheorem}{C.\arabic{theorem}}
\section{\texorpdfstring{\ann{Spherical-membrane inflation: pressure inversion and turning points}}{Spherical-membrane inflation: pressure inversion and turning points}}

Under uniform inflation of an incompressible spherical membrane, a single surface stretch $\lambda$ determines the three principal stretches $(\lambda^{-2},\lambda,\lambda)$. The traction-free condition in the thickness direction fixes the incompressibility multiplier, after which the Laplace relation converts the hoop stress into an internal pressure $p_{\mathrm{sph}}(\lambda)$. These thin-wall kinematics and the pressure balance are discussed in \citet[Section~3.1.1]{Voss2021Thesis}. Thin balloons often exhibit a pressure maximum at moderate inflation followed by a jump under pressure control \citep[Section~3.1.6]{Voss2021Thesis}. Incompressible neo-Hooke provides an explicit example: $p_{\mathrm{sph}}\propto(\lambda^{-1}-\lambda^{-7})$ has a unique maximum at $\lambda=7^{1/6}$. Convexity with respect to $I_1$, or polyconvexity of the three-dimensional incompressible law, does not by itself exclude this limit point.

Consider the incompressible Demiray family
\begin{equation}
W_{\mathrm D,b}(F)
=\frac{\mu}{2b}\left[\exp\bigl(b(I_1-3)\bigr)-1\right],
\qquad b>0,\qquad J=1.
\end{equation}
It contains two energies from Appendix~B after isochoric restriction: $W_k|_{J=1}=W_{\mathrm D,b}$ with $b=k$, and $W_{\exp}|_{J=1}=W_{\mathrm D,1}$. The exponentiated Hencky energy is outside this $I_1$ family but also has an invertible spherical-membrane pressure for $k_{\mathrm{eH}}>3/8$. Consequently, spherical-membrane pressure inversion is shared by these constitutive laws. Set $b_*:=1/(8\sqrt3)=k_*$. Section~C.1 derives the Demiray pressure law, and Section~C.2 proves global pressure inversion for every $b\ge b_*$. \ann{Sections~C.3--C.6 then determine when turning points occur. Section~C.3 gives the sharp Demiray threshold, Section~C.4 treats the exponentiated Hencky energy, Section~C.5 derives a general pressure-growth criterion, and Section~C.6 constructs a compressible energy that retains the global constitutive properties of $W_k$ and has a descending pressure branch.}

\begin{figure}[ht]
\centering
\includegraphics[width=\textwidth]{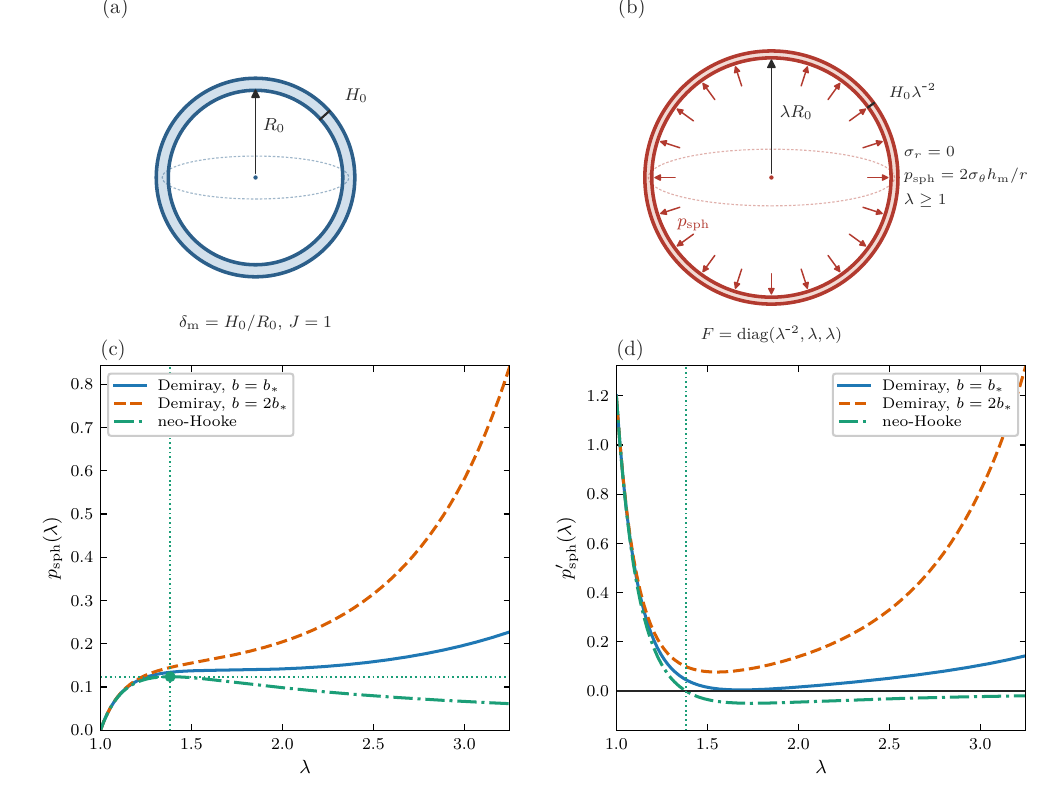}
\caption{Kinematics, limit point, and pressure-stretch inversion for the incompressible Demiray family. Panel (a) shows the reference mid-surface radius $R_0$, thickness $H_0$, and thickness ratio $\delta_{\mathrm m}=H_0/R_0$. Panel (b) shows the isochoric inflation, with $F=\operatorname{diag}(\lambda^{-2},\lambda,\lambda)$, current radius $\lambda R_0$, and current thickness $h_{\mathrm m}=H_0\lambda^{-2}$; the thickness traction vanishes, $\sigma_r=0$, and the Laplace relation is $p_{\mathrm{sph}}=2\sigma_\theta h_{\mathrm m}/r$. In panel (c), solid blue and dashed orange are the Demiray pressures at $b=b_*$ and $b=2b_*$, and dash-dotted green is the incompressible neo-Hooke pressure under the same spherical kinematics. \ann{The neo-Hooke curve reaches its marked pressure maximum at $\lambda=7^{1/6}$, while both displayed Demiray curves increase throughout.} The dotted horizontal and vertical lines identify \ann{the neo-Hooke maximum}. Panel (d) shows the corresponding pressure derivatives: both Demiray curves remain positive, whereas the neo-Hooke derivative becomes negative beyond the limit point. The parameters are $\sigma_{\mathrm{ref}}=\mu/2=1$ and $\delta_{\mathrm m}=0.05$. Theorem C.1 proves that, for every $b\ge b_*$, $p_{\mathrm{sph}}'(\lambda)>0$ throughout $\lambda\ge1$ and the pressure increases from zero to infinity. Since $W_k|_{J=1}$ corresponds to $b=k$ and $W_{\exp}|_{J=1}$ to $b=1$, the theorem applies to both restricted energies.}
\label{fig:C1}
\end{figure}

\subsection{Stress and pressure under isochoric spherical deformation}

Appendix B completed the pressure inversion for the annular boundary-value problem. We now work directly with $W_{\mathrm D,b}$ and derive its hoop Cauchy stress and internal pressure from the traction-free condition in the thickness direction and the Laplace relation. The identities

\begin{equation}
\left.W_k(F)\right|_{J=1}
=W_{\mathrm D,k}(F),\qquad
\left.W_{\exp}(F)\right|_{J=1}
=W_{\mathrm D,1}(F)
\end{equation}
place both restricted energies in the same family \citep{Fung1967,Demiray1972}. Explicitly,
\begin{equation}
W_{\mathrm D,b}(F)
=\frac{\mu}{2b}
\left[
\exp\!\left(b(I_1-3)\right)-1
\right].
\end{equation}

Consider a thin spherical membrane with reference radius $R_0$, reference thickness $H_0$, and

\begin{equation}
F=\operatorname{diag}(\lambda^{-2},\lambda,\lambda),
\qquad
\lambda\ge1,
\qquad
\delta_{\mathrm m}=\frac{H_0}{R_0}>0.
\end{equation}
\ann{The dimensionless thickness ratio is $\delta_{\mathrm m}$, and $p_{\mathrm{sph}}(\lambda)$ denotes the membrane pressure.} Figure~\ref{fig:C1}(a) shows this reference membrane.
\ann{Then $I_1=2\lambda^2+\lambda^{-4}$.}
Eliminating the incompressibility multiplier with the traction-free condition in the thickness direction gives

\begin{equation}
\sigma_\theta
=\mu
\exp\!\left(b(2\lambda^2+\lambda^{-4}-3)\right)
(\lambda^2-\lambda^{-4}).
\end{equation}
The current thickness is $H_0\lambda^{-2}$ and the current radius is $\lambda R_0$. The Laplace pressure is therefore

\begin{equation}
p_{\mathrm{sph}}(\lambda)
=2\mu\delta_{\mathrm m}(\lambda^{-1}-\lambda^{-7})
\exp\!\left(b(2\lambda^2+\lambda^{-4}-3)\right).
\end{equation}
Figure~\ref{fig:C1}(b) shows the inflated membrane and the Laplace balance that produces this pressure. Figure~\ref{fig:C1}(c) plots the same pressure at $b=b_*$ and $b=2b_*$, together with the neo-Hooke limit $p_{\mathrm{NH}}(\lambda)=2\mu\delta_{\mathrm m}(\lambda^{-1}-\lambda^{-7})$.
\ann{Equivalently, writing $\widehat W_{\mathrm D,b}(\lambda)=W_{\mathrm D,b}(\operatorname{diag}(\lambda^{-2},\lambda,\lambda))$, we have $p_{\mathrm{sph}}(\lambda)=\delta_{\mathrm m}\lambda^{-2}\widehat W_{\mathrm D,b}^{\prime}(\lambda)$.}
For comparison, the exponentiated Hencky energy gives under the same kinematics
\begin{equation}
p_{\mathrm{eH}}(\lambda)
=12\mu\delta_{\mathrm m}\lambda^{-3}
\exp\!\left(6k_{\mathrm{eH}}(\log\lambda)^2\right)\log\lambda.
\end{equation}
With $t=\log\lambda$, the sign of its derivative is the sign of ${1-3t+12k_{\mathrm{eH}}t^2}$, which is positive for every $t\ge0$ when $k_{\mathrm{eH}}>3/8$. Its pressure starts from zero and diverges as $\lambda\to\infty$. This model therefore shares the spherical inversion while lying outside the Demiray family and lacking global rank-one convexity \citep{Neff2015ExpHencky}.
\subsection{Strict pressure increase along the inflation branch}

It remains to prove $p_{\mathrm{sph}}'(\lambda)>0$ for $\lambda\ge1$ and $p_{\mathrm{sph}}(\lambda)\to\infty$; together with $p_{\mathrm{sph}}(1)=0$, these properties imply global pressure-stretch inversion.

\begin{theorem}\label{thm:C1} Let $\mu>0$, $\delta_{\mathrm m}>0$, and $b\ge b_*$. For the incompressible Demiray energy $W_{\mathrm D,b}$, every prescribed internal pressure $p_*\ge0$ determines exactly one spherical inflation stretch $\lambda\ge1$, and that stretch depends smoothly on $p_*$. \ann{Equivalently, $p_{\mathrm{sph}}:[1,\infty)\to[0,\infty)$}
is a smooth, strictly increasing bijection whose inverse admits a one-sided smooth extension at $p_{\mathrm{sph}}=0$.
\end{theorem}

\ann{\noindent\textit{Proof.} Put $e(\lambda)=\exp(b(2\lambda^2+\lambda^{-4}-3))$.}
For every $\lambda>0$,

\begin{equation}
\ann{p_{\mathrm{sph}}^{\prime}(\lambda)=2\mu\delta_{\mathrm m}e(\lambda)\lambda^{-12}\Phi_{\mathrm{sph}}(\lambda,b),\qquad
\Phi_{\mathrm{sph}}(\lambda,b)=7\lambda^4-\lambda^{10}+4b(\lambda^6-1)^2.}
\end{equation}
\ann{At $\lambda=1$, it holds that $p_{\mathrm{sph}}^{\prime}(1)=12\mu\delta_{\mathrm m}>0$.}
For $1<\lambda\le7^{1/6}$, both $\lambda^4(7-\lambda^6)$ and the $b$ term are nonnegative, and their sum is strictly positive. For $\lambda>7^{1/6}$, $\Phi_{\mathrm{sph}}$ is increasing in $b$, so it suffices to set $b=b_*$. With $\vartheta_\lambda:=\lambda^2\ge1$, the desired inequality is

\begin{equation}
(\vartheta_\lambda^3-1)^2>2\sqrt3\,\vartheta_\lambda^2(\vartheta_\lambda^3-7).
\end{equation}
Both sides are positive in this range. Squaring and subtracting gives

\begin{equation}
S_{\mathrm m}(\vartheta_\lambda)
=(\vartheta_\lambda^3-1)^4-12\vartheta_\lambda^4(\vartheta_\lambda^3-7)^2.
\end{equation}
Exact polynomial division yields

\begin{equation}
\ann{\begin{aligned}
S_{\mathrm m}(\vartheta_\lambda)
&=(\vartheta_\lambda^3-7)^2M_{\mathrm m}(\vartheta_\lambda)+864(\vartheta_\lambda^3-7)+1296,\\
M_{\mathrm m}(\vartheta_\lambda)
&=\vartheta_\lambda^6-12\vartheta_\lambda^4+10\vartheta_\lambda^3+97
=\vartheta_\lambda^2(\vartheta_\lambda^2-6)^2+K_{\mathrm m}(\vartheta_\lambda),\\
K_{\mathrm m}(\vartheta_\lambda)&=10\vartheta_\lambda^3-36\vartheta_\lambda^2+97.
\end{aligned}}
\end{equation}
On $\vartheta_\lambda\ge0$, the minimum of $K_{\mathrm m}$ occurs at $\vartheta_\lambda=12/5$, \ann{and $K_{\mathrm m}(12/5)=697/25>0$.}
Thus $M_{\mathrm m}(\vartheta_\lambda)>0$. \ann{For $\vartheta_\lambda^3\ge7$, we have $S_{\mathrm m}(\vartheta_\lambda)\ge1296>0$.}
\ann{Consequently, $\Phi_{\mathrm{sph}}(\lambda,b)>0$ and $p_{\mathrm{sph}}^{\prime}(\lambda)>0$ for $\lambda\ge1$.}
Figure~\ref{fig:C1}(d) shows this derivative at $b=b_*$ and $b=2b_*$, together with $p_{\mathrm{NH}}'(\lambda)=2\mu\delta_{\mathrm m}(-\lambda^{-2}+7\lambda^{-8})$. Both Demiray curves stay positive on $1\le\lambda\le3.25$. The neo-Hooke derivative changes sign at $\lambda=7^{1/6}\ann{\approx\ann{1.383}}$, where the pressure reaches the limit-point value $\ann{0.124}$.
Finally,

\begin{equation}
p_{\mathrm{sph}}(1)=0,
\qquad
p_{\mathrm{sph}}(\lambda)
\sim2\mu\delta_{\mathrm m} e^{-3b}\lambda^{-1}e^{2b\lambda^2}
\to\infty
\quad(\lambda\to\infty).
\end{equation}
Figure~\ref{fig:C1}(c) shows the corresponding pressures, computed at $\sigma_{\mathrm{ref}}=\mu/2=1$ and $\delta_{\mathrm m}=0.05$. The Demiray curves increase strictly from the identity; at $\lambda=3.25$ the values are $\ann{0.228}$ at $b=b_*$ and $\ann{0.842}$ at $b=2b_*$, with derivative $\ann{0.144}$ at $b=b_*$. The neo-Hooke pressure has already fallen to $\ann{0.061}$ with derivative $-\ann{0.019}$, so a pressure below the peak can correspond to two stretches. For every $\lambda>1$ the explicit formula is strictly increasing in $b$.
Strict monotonicity, endpoint growth, and the inverse-function theorem prove the claim. $\square$

Thus the entire Demiray family with $b\ge b_*$ has a globally invertible spherical inflation pressure, whereas its neo-Hooke limit has a descending branch and two stretches for pressures below the peak. The result applies to $W_k|_{J=1}$ and $W_{\exp}|_{J=1}$, and the closed calculation in Section~C.1 gives the same inversion for $W_{\mathrm{eH}}$. These hardened constitutive laws therefore share monotone spherical pressure growth, in contrast to the neo-Hooke limit point.
\par
\subsection{\ann{Sharp transition within the Demiray family}}\label{sec:balloon-threshold}

\ann{The monotonicity in Theorem C.1 results from the specified parameter range. The same Demiray family has a pressure maximum followed by a minimum when its stiffening parameter is sufficiently small. The sign of $p_{\mathrm{sph}}'$ is the sign of $\Phi_{\mathrm{sph}}(\lambda,b)$ from the proof of Theorem C.1. Since $\Phi_{\mathrm{sph}}>0$ for $1\le\lambda\le7^{1/6}$, turning points can occur only for $\lambda>7^{1/6}$. Put $\zeta_\lambda=\lambda^6>7$. Solving $\Phi_{\mathrm{sph}}(\lambda,b)=0$ for $b$ gives}
\ann{\begin{equation}
b=\beta_{\mathrm D}(\zeta_\lambda):=\frac{(\zeta_\lambda-7)\zeta_\lambda^{2/3}}{4(\zeta_\lambda-1)^2},\qquad
\beta_{\mathrm D}'(\zeta_\lambda)=-\frac{\zeta_\lambda^2-23\zeta_\lambda-14}{12\zeta_\lambda^{1/3}(\zeta_\lambda-1)^3}.
\end{equation}}%
\ann{Since $\beta_{\mathrm D}(7)=0$, $\beta_{\mathrm D}(\zeta_\lambda)\to0$ as $\zeta_\lambda\to\infty$ and $\beta_{\mathrm D}'$ changes sign only at $\zeta_{\lambda,c}$, the unique maximum of $\beta_{\mathrm D}$ occurs at}
\ann{\begin{equation}\label{eq:balloon-critical}
\zeta_{\lambda,c}=\frac{23+\sqrt{585}}2,\qquad
b_c=\beta_{\mathrm D}(\zeta_{\lambda,c})\approx \ann{0.067},
\qquad \lambda_{\mathrm D,c}=\zeta_{\lambda,c}^{1/6}\approx \ann{1.694}.
\end{equation}}%
\ann{Since $\Phi_{\mathrm{sph}}$ increases with $b$, the pressure decreases exactly where $b<\beta_{\mathrm D}(\lambda^6)$. For $0<b<b_c$, the equation $\beta_{\mathrm D}(\zeta_\lambda)=b$ has two roots. They give a pressure maximum and a subsequent minimum, followed by an unbounded second ascending branch. For $b>b_c$, $p_{\mathrm{sph}}'(\lambda)>0$ everywhere on $[1,\infty)$. At $b=b_c$, the derivative vanishes only at $\lambda_{\mathrm D,c}$. The pressure is still strictly increasing, but its inverse is not differentiable there. The neo-Hooke limit $b=0$ retains only the maximum, with pressure tending to zero at large stretch. In particular,}
\ann{\begin{equation}
b_c<k_*=\frac1{8\sqrt3}\approx \ann{0.072}.
\end{equation}}%
\ann{Thus the global \TSTSM{} range of the compressible $W_k$ lies entirely in the monotone spherical-membrane regime of its incompressible restriction. Figure C1 displays this range: its marked limit point belongs solely to the neo-Hooke comparison curve.}

\ann{The Demiray family thus gives pressure turning points using $I_1$ alone. The generalized neo-Hookean limiting-chain model of \citet{AnssariBenamBucchiSaccomandi2022Inflation} also has a pressure maximum, a subsequent minimum and a second ascending branch. For the uniform spherical equilibria considered here, the hardening function determines whether the pressure has two turning points or increases monotonically.}
\par
\subsection{\ann{Exponentiated Hencky: spherical and annular responses}}\label{sec:hencky-balloon}

\ann{For the exponentiated Hencky energy, the derivative of the spherical pressure $p_{\mathrm{eH}}$ in Section C.1 has the sign of ${1-3t+12k_{\mathrm{eH}}t^2}$, where $t=\log\lambda\ge0$. Its discriminant is $9-48k_{\mathrm{eH}}$. Consequently, two spherical pressure turning points occur for $0<k_{\mathrm{eH}}<3/16$, at}
\ann{\begin{equation}
t_\pm=\frac{3\pm\sqrt{9-48k_{\mathrm{eH}}}}{24k_{\mathrm{eH}}}.
\end{equation}}%
\ann{At $k_{\mathrm{eH}}=3/16$ the derivative vanishes only at $t=2/3$, and for $k_{\mathrm{eH}}>3/16$ it is strictly positive. The three-dimensional global \TSTSM{} range used here, $k_{\mathrm{eH}}>3/8$, therefore gives increasing spherical pressure. At $k_{\mathrm{eH}}=0.40$, the same energy gives the nonmonotone plane-strain annular branch computed in Appendix B. The spherical and annular calculations show how the geometry changes the pressure response of the same constitutive law.}
\par
\subsection{\ann{Stress growth and spherical-membrane pressure}}\label{sec:own-i1-balloon}

\ann{Pressure decrease is governed by the growth of the stress along the spherical loading path. For a $C^2$ isotropic energy restricted to $F=\operatorname{diag}(\lambda^{-2},\lambda,\lambda)$, write $\widehat W(\lambda)$ for the restricted energy and $\Sigma=\sigma_\theta-\sigma_r$. The incompressibility multiplier does not affect this stress difference. The work-conjugacy identity and the thin-membrane balance give}
\ann{\begin{equation}\label{eq:general-balloon-criterion}
\Sigma=\frac\lambda2\widehat W'(\lambda),\qquad
p_{\mathrm{sph}}=2\delta_{\mathrm m}\lambda^{-3}\Sigma=\delta_{\mathrm m}\lambda^{-2}\widehat W'(\lambda),\qquad
p_{\mathrm{sph}}'=\delta_{\mathrm m}\lambda^{-3}\bigl(\lambda\widehat W''-2\widehat W'\bigr).
\end{equation}}%
\ann{Consequently, at positive pressure,}
\ann{\begin{equation}\label{eq:balloon-stress-growth}
p_{\mathrm{sph}}'<0\quad\Longleftrightarrow\quad
\frac{\mathrm d\log\Sigma}{\mathrm d\log\lambda}<3.
\end{equation}}%
\ann{The stress difference can increase throughout a descending pressure branch. Its growth is then slower than the factor $\lambda^3$ associated with membrane thinning and radius increase. With $t=\log\lambda$ and $N_{\mathrm{sph}}=\operatorname{diag}(-2,1,1)$, the spherical path is $X=tN_{\mathrm{sph}}$. Since $\operatorname{tr}N_{\mathrm{sph}}=0$, the incompressibility multiplier drops out of the pairing, and global \TSTSM{} gives $\langle\mathrm{D}_X\widehat\sigma(tN_{\mathrm{sph}}).N_{\mathrm{sph}},N_{\mathrm{sph}}\rangle=2\,\mathrm d\Sigma/\mathrm dt>0$. For an energy already satisfying global \TSTSM{}, the spherical pressure decreases exactly when $0<\mathrm d\Sigma/\mathrm dt<3\Sigma$. For an incompressible energy $f(I_1)$ with $f'>0$, the pressure criterion becomes}
\ann{\begin{equation}
p_{\mathrm{sph}}'<0\quad\Longleftrightarrow\quad
\frac{f''(I_1)}{f'(I_1)}
<\frac{\lambda^4(\lambda^6-7)}{4(\lambda^6-1)^2},\qquad
I_1=2\lambda^2+\lambda^{-4},\quad\lambda>1.
\end{equation}}%
\ann{For the Demiray energy, $f''/f'=b$ and the right-hand side equals $\beta_{\mathrm D}(\lambda^6)$ from Section C.3, so the criterion reproduces the threshold derived there.}

\par
\subsection{\ann{A compressible energy with descending spherical pressure}}\label{sec:raw-invariant-balloons}

\ann{Section C.5 shows how increasing material stress can produce decreasing spherical pressure. For $W_k$ this response does not occur in the global \TSTSM{} range, because its incompressible restriction has $b=k$, while \TSTSM{} requires $k\ge k_*>b_c$. The energy below changes two coefficients. Halving the coefficient of $I_1$ gives an incompressible restriction with $b=k/2$. The modified volume terms make the volumetric curvature with respect to $s=\log J$ at least $2$, which lowers the sufficient \TSTSM{} threshold to $k\ge1/8$. The admissible value $k=1/8$ therefore gives $b=1/16<b_c$, and the reference moduli remain those of $W_k$. The energy is defined for $\mu>0$ and $k>0$. For $k\ge1/8$, we prove global \TSTSM{}, including the endpoint, together with polyconvexity and stress bijectivity. Consider}
\ann{\begin{equation}\label{eq:raw-invariant-energies}
Q^{(1)}(F)=\frac{I_1-3}{2}-\log J+J+J^{-1}-2,
\qquad
W^{(1)}_k(F)=\frac\mu k\left(e^{kQ^{(1)}(F)}-1\right).
\end{equation}}%
\ann{This energy is smooth, objective and isotropic on all of $\mathrm{GL}^+(3)$. The complete tensor argument is given by \eqref{eq:raw-tsts-completion}. We first establish its three-dimensional constitutive properties, including the endpoint $k=1/8$, and then impose incompressibility to calculate the spherical-membrane pressure.}

\ann{To determine the reference response, set $X=\log V$, $s=\operatorname{tr}X$ and ${\widehat Q^{(1)}(X):=Q^{(1)}(e^X)=g^{(1)}(X)+v^{(1)}(s)}$, where}
\ann{\begin{equation}
g^{(1)}(X)=\frac12\left(\operatorname{tr}e^{2X}-3\right),
\qquad
v^{(1)}(s)=-s+e^s+e^{-s}-2.
\end{equation}}%
\ann{The first derivatives of these two terms cancel at $X=0$, so $\widehat Q^{(1)}(0)=0$ and $\Diff_X\widehat Q^{(1)}(0)=0$. Their second-order expansion gives}
\ann{\begin{equation}
\widehat Q^{(1)}(X)=\lVert\operatorname{dev}X\rVert^2+\frac43s^2+O(\lVert X\rVert^3).
\end{equation}}%
\ann{Thus the reference state is stress free, with $\mu_0=\mu$, $\kappa_0=8\mu/3$ and $\nu_0=1/3$, and has the same shear and bulk moduli as $W_k$. These moduli are independent of $k$, because the exponential outer function contributes $W^{(1)}_k=\mu \widehat Q^{(1)}+O((\widehat Q^{(1)})^2)$ near the reference state.}

\ann{Since $I_1=\lVert F\rVert^2$ is quadratic in $F$, polyconvexity and ellipticity follow directly. The independent-minor representative $\mathcal Q^{(1)}(Y,G,\ann{J})=\tfrac12(\lVert Y\rVert^2-3)-\log \ann{J}+\ann{J}+\ann{J}^{-1}-2$, with $\ann{J}>0$, satisfies $Q^{(1)}(F)=\mathcal Q^{(1)}(F,\Cof F,\det F)$. It is convex because its $Y$-term is quadratic and the second derivative in $\ann{J}$ is $\ann{J}^{-2}+2\ann{J}^{-3}>0$. The convex increasing exponential therefore gives a polyconvex energy. \ann{For $L=\xi_F\otimes\eta_F\ne0$, the determinant is affine along the rank-one line, so} $\Diff_FJ(F)=\Cof F$ and $\langle\Diff_F^2J(F).L,L\rangle=0$. Hence}
\ann{\begin{equation}
\left\langle\Diff_F^2Q^{(1)}(F).L,L\right\rangle
=\lVert L\rVert^2+(J^{-2}+2J^{-3})\left\langle\Diff_FJ(F),L\right\rangle^2>0.
\end{equation}}%
\ann{Differentiating the exponential at the same $F$ then gives}
\ann{\begin{equation}
\left\langle\Diff_F^2W^{(1)}_k(F).L,L\right\rangle
=\mu e^{kQ^{(1)}}\left\{
\left\langle\Diff_F^2Q^{(1)}(F).L,L\right\rangle
+k\left\langle\Diff_FQ^{(1)}(F),L\right\rangle^2\right\}>0.
\end{equation}}%
\ann{Strict Legendre-Hadamard positivity therefore holds for every $k>0$.}

\ann{The TSTS condition concerns increments of the logarithmic strain. Here $s=\operatorname{tr}X=\operatorname{tr}(\log V)=\log J$. The function $g^{(1)}(X)=\tfrac12(\operatorname{tr}e^{2X}-3)$ depends on the full logarithmic strain. Its Hessian is positive also for nonzero spherical increments, as the following formula shows. In a principal basis of $X$, let $x_i$ be its eigenvalues and let $Z_{ij}$ be the components of a symmetric increment $Z$. The spectral Hessian formula \citep{LewisSendov2001} gives}
\ann{\begin{equation}
\left\langle\Diff_X^2g^{(1)}(X).Z,Z\right\rangle
=2\sum_i e^{2x_i}Z_{ii}^2
+2\sum_{i<j}\frac{e^{2x_i}-e^{2x_j}}{x_i-x_j}Z_{ij}^2.
\end{equation}}%
\ann{At repeated eigenvalues the divided difference is interpreted by its positive limit $2e^{2x_i}$. It follows that}
\ann{\begin{equation}
\left\langle\Diff_X^2g^{(1)}(X).Z,Z\right\rangle\ge2e^{2\min_i x_i}\lVert Z\rVert^2>0
\qquad (Z\ne0).
\end{equation}}%
\ann{This estimate covers the full symmetric-matrix space. The Cauchy stress is ${\widehat\sigma^{(1)}_k(X)=\mu e^{k\widehat Q^{(1)}(X)-s}\Diff_X\widehat Q^{(1)}(X)}$. Differentiation and completion of the square yield the following identity, with all quantities evaluated at the current state:}
\ann{\begin{equation}\label{eq:raw-tsts-completion}
\frac{e^{s-k\widehat Q^{(1)}}}\mu
\left\langle\Diff_X\widehat\sigma^{(1)}_k.Z,Z\right\rangle
=\left\langle\Diff_X^2g^{(1)}.Z,Z\right\rangle
+\left((v^{(1)})''(s)-\frac1{4k}\right)(\operatorname{tr}Z)^2
+k\left[\left\langle\Diff_X\widehat Q^{(1)},Z\right\rangle-\frac{\operatorname{tr}Z}{2k}\right]^2.
\end{equation}}%
\ann{Since $(v^{(1)})''(s)=e^s+e^{-s}\ge2$, the last two terms are nonnegative for $k\ge1/8$, while the first term is strictly positive for every $Z\ne0$. At $k=1/8$ and $s=0$, the middle coefficient is zero, but the first term remains strictly positive. Thus the endpoint is included: the complete Cauchy-stress tangent satisfies strict \TSTSM{} for every finite $X$. The condition $k\ge1/8$ is sufficient.}

\ann{To obtain global stress inversion, it remains to show that unbounded logarithmic strain cannot produce bounded stress. The function $\widehat Q^{(1)}$ is convex, and $\widehat Q^{(1)}(0)=0$ and $\Diff_X\widehat Q^{(1)}(0)=0$ imply $\widehat Q^{(1)}\ge0$ and ${\langle\Diff_X\widehat Q^{(1)}(X),X\rangle\ge \widehat Q^{(1)}(X)}$. Set $\mathcal R_Q=\widehat Q^{(1)}(X)+7/2$. Using $|s|\le\cosh s$ gives}
\ann{\begin{equation}
\mathcal R_Q=\frac12\sum_i e^{2x_i}-s+e^s+e^{-s}
\ge\frac12\sum_i e^{2x_i}+\frac12(e^s+e^{-s}).
\end{equation}}%
\ann{Consequently, $e^{2x_i}\le2\mathcal R_Q$, $e^{\pm s}\le2\mathcal R_Q$ and $x_i=s-x_j-x_\ell$ imply}
\ann{\begin{equation}
-2\log(2\mathcal R_Q)\le x_i\le\frac12\log(2\mathcal R_Q),
\qquad
\lVert X\rVert\le2\sqrt3\log(2\mathcal R_Q).
\end{equation}}%
\ann{Thus $\lVert X\rVert\to\infty$ forces $\widehat Q^{(1)}\to\infty$. The preceding convexity estimate now gives}
\ann{\begin{equation}
\lVert\widehat\sigma^{(1)}_k(X)\rVert
\ge\mu e^{k\widehat Q^{(1)}-s}\frac{\widehat Q^{(1)}}{\lVert X\rVert}
\ge\frac{\mu e^{k\widehat Q^{(1)}}\widehat Q^{(1)}}{2\mathcal R_Q\lVert X\rVert}
\longrightarrow\infty.
\end{equation}}%
\ann{The stress map is therefore proper. For $k\ge1/8$, strict \TSTSM{} supplies injectivity and a nonsingular tangent. The open-and-closed image argument of Section \ann{3.5} then gives surjectivity and a global $C^1$ stress inverse.}

\ann{We can now evaluate the spherical response of an energy with these global properties. Choose $k=1/8$ and impose $J=1$ as an additional kinematic constraint. The volume contribution vanishes on this path, so ${W^{(1)}_{1/8}|_{J=1}=8\mu[e^{(I_1-3)/16}-1]}$. Its spherical-membrane pressure, denoted $p_1=p_{\mathrm{sph}}$ for this example, is}
\ann{\begin{equation}\label{eq:raw-i1-balloon}
p_1(\lambda)=2\mu\delta_{\mathrm m}(\lambda^{-1}-\lambda^{-7})
 e^{(2\lambda^2+\lambda^{-4}-3)/16}.
\end{equation}}%
\ann{This is the Demiray membrane with $b=1/16<b_c$, so Section C.3 proves that it has exactly two pressure turning points. At $\lambda=\sqrt3$, a direct calculation gives ${\lambda p_1'/p_1=-11/234<0}$. The pressure maximum and minimum occur at $\lambda\approx1.5752$ and $1.8750$, with ${p_1/(2\mu\delta_{\mathrm m})\approx0.67755}$ and $0.67376$, respectively. The peak-to-trough pressure decrease is approximately $0.559\%$.}

\ann{Figure~\ref{fig:C2} shows this descending pressure interval and the increasing hoop-radial stress difference $\Sigma=p_1\lambda^3/(2\delta_{\mathrm m})$. As established in Section C.5, the decrease occurs because $\Sigma$ grows more slowly than $\lambda^3$. The modified volume term enters only through this TSTS threshold, since it does not contribute to the pressure under $J=1$. This example therefore exhibits descending spherical pressure while retaining full-domain polyconvexity, strict ellipticity, strict \TSTSM{} and stress bijectivity.}
\ann{
\begin{figure}[htbp]
\centering
\includegraphics[width=\textwidth]{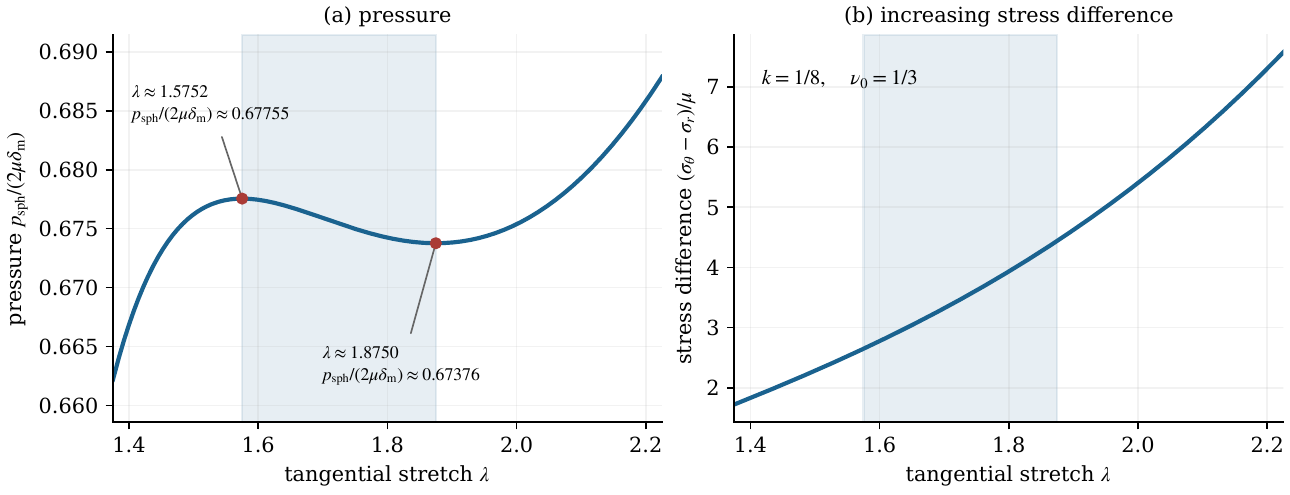}
\annfigurelabel
\caption{\ann{Spherical inflation for the compressible energy $W^{(1)}_{1/8}$, with $\nu_0=1/3$, under the additional incompressibility constraint. The blue curves show (a) dimensionless pressure and (b) the hoop-radial stress difference. Red points mark the pressure maximum at $\lambda\approx1.5752$ and minimum at $\lambda\approx1.8750$, and shading marks the descending pressure interval. The enlarged pressure scale resolves the peak-to-trough decrease of approximately $0.559\%$. The stress difference continues to increase over the same interval. The complete three-dimensional energy remains polyconvex, strictly elliptic and globally \TSTSM{}, with a globally invertible stress map.}}
\label{fig:C2}
\end{figure}
}

\FloatBarrier
\section*{Acknowledgments}
The stored-energy family $W_k$ and the two threshold hypotheses were first isolated by a symbolic-search program written by the authors.
The authors used OpenAI GPT-5.6 to review proof drafts, to check code, and to assist with language editing.
The authors reviewed all outputs and take full responsibility for the published article.

We have found $W_k$ on 31 July 2026; the full checks have been accomplished by 15 August 2026.
In the process of submitting the paper, the authors were informed that another solution to the challenge by Neff has been obtained independently by different means and with different techniques by a group from Stanford University \citep{XieJaviliNeffLinder2026}.

\newann{The prize offered by Patrizio Neff for this challenge motivated this work and has been given to the first author.}

K.Z.\ acknowledges support from the National Natural Science Foundation of China under grant No.~12372173 and the Natural Science Foundation of Shanghai under grant No.~23ZR1468600. Y.L.\ thanks the Natural Science Foundation of Shanghai under grant No.~24ZR1425400.

\section*{Declaration of competing interest}
The authors declare that they have no known competing financial interests or personal relationships that could have appeared to influence the work reported in this paper.

% The supplied bibliography preserves the checked citation labels and formatting.

\end{document}